\documentclass[11pt]{article}
\usepackage[margin=1.2in]{geometry}

\DeclareFontFamily{OT1}{ptm}{}

\usepackage{latexsym}
\usepackage{amsmath}
\usepackage{amssymb}
\usepackage{mathrsfs}
\usepackage{color}
\usepackage{tikz}
\usepackage{caption}
\usepackage{subcaption}

\usepackage{graphicx}
\usetikzlibrary{cd}
\usetikzlibrary{shapes,decorations,arrows,calc,arrows.meta,fit,positioning}
\tikzstyle{startstop} = [rectangle, rounded corners, minimum width=0.5cm, minimum height=0.5cm,text centered, text width=2cm, draw=black, fill=white!30]
\tikzstyle{arrow} = [thick,->,>=stealth]

\usepackage{enumitem}
\usepackage{amsthm}
\newtheorem{theorem}{Theorem}[section]
\newtheorem{thm}[theorem]{Theorem}

\newtheorem{lemma}[theorem]{Lemma}
\newtheorem{conjecture}[theorem]{Conjecture}

\newtheorem{question}[theorem]{Question}

\newtheorem{definition}[theorem]{Definition}

\newtheorem{rmk}[theorem]{Remark}

\newtheorem{example}[theorem]{Example}

\newtheorem{prop}[theorem]{Proposition}
\newcommand{\mbf}{\mathbf}
\newcommand{\mb}{\mathbb}
\newcommand{\mc}{\mathcal}
\newcommand{\ZZ}{\mathbb{Z}}

\newcommand{\CC}{\mathbb{C}}
\newcommand{\RR}{\mathbb{R}}
\newcommand{\PP}{\mathbb{P}}

\usepackage[toc,page,title,titletoc,header]{appendix}
\usepackage{hyperref}
\hypersetup{
	backref,colorlinks=true,linkcolor=blue,citecolor=blue,urlcolor=blue,citebordercolor={0 0 1},urlbordercolor={0 0 1},linkbordercolor={0 0 1},
}
\title{Local geometry of symplectic divisors with applications to contact torus bundles}
\author{Tian-Jun Li, Jie Min}
\AtEndDocument{\bigskip{\footnotesize
		\textsc{School of Mathematics, University of Minnesota, Minneapolis, MN, US} \par  
		\textit{E-mail address}: \texttt{tjli@math.umn.edu} \par
		\addvspace{\medskipamount}
		\textsc{School of Mathematics, University of Minnesota, Minneapolis, MN, US} \par  
		\textit{E-mail address}: \texttt{minxx127@umn.edu} \par
		
}}
\begin{document}
\maketitle

\begin{abstract}
	In this note we study the contact geometry of symplectic divisors. We show the contact structure induced on the boundary of a divisor neighborhood is invariant under toric and interior blow-ups and blow-downs. We also construct an open book decomposition on the boundary of a concave divisor neighborhood and apply it to the study of universally tight contact structures of contact torus bundles.
\end{abstract}

\tableofcontents

\section{Introduction}

A \textbf{topological divisor} $ D $ refers to a connected configuration of finitely many closed embedded oriented smooth surfaces $ D=C_1\cup \dots \cup C_k $ in a smooth oriented 4-manifold $ X $ (possibly with boundary or non-compact).
In this paper, a topological divisor $ D $ is required to satisfy the following additional properties: $ D $ doesn't intersect the boundary of $ X $, no three $ C_i's $ intersect at the same point, and any intersection between two surfaces is positive and transversal. 
In a symplectic 4-manifold $(X,\omega)$ (possibly with boundary or non-compact), a \textbf{symplectic divisor} is a topological divisor $D$ embedded in $X$, with each component being a symplectic surface and having the positive orientation with respect to $\omega$.
Since we are interested in the germ of a symplectic divisor, $ X $ is sometimes omitted in the writing and $ (D,\omega ) $, or simply $ D $, is used to denote a symplectic divisor.

Given a divisor $ D=\cup_{i=1}^k C_i $ in $ (X,\omega ) $, the intersection matrix of $ D $ is a $ k\times k $ matrix $ Q_D=([C_i]\cdot [C_j]) $, where $ \cdot  $ is used for any of the pairings $ H_2(X;\mb{K})\times H_2(X;\mb{K}) $, $ H^2(X;\mb{K})\times H_2(X;\mb{K}) $ and $ H^2(X;\mb{K})\times H^2(X,\partial X;\mb{K}) $. Here $ \mb{K}$ could be either $\ZZ$ or $\RR  $ depending on the situation. We also denote by $b^+(D)$ the number of positive eigenvalues of $Q_D$.

Let $ (D=\cup_{i=1}^k,\omega) $ be a symplectic divisor. 
A symplectic divisor is called $ \omega $-orthogonal is any two components intersect $ \omega $-orthogonally. 
A closed regular neighborhood of $ D $ is called a plumbing of $ D $. A plumbing $ N_D $ of $ D $ is called a \textbf{concave plumbing} (resp. \textbf{convex plumbing}) if it is a strong symplectic cap (resp. filling) of its boundary $(-Y_D,\xi_D)$ (resp. $(Y_D,\xi_D)$). A concave plumbing is also called a \textbf{divisor cap} of its boundary. Let $ Q_D $ be the intersection matrix of $ D $ and $ a=(C_i\cdot [\omega])\in (\RR_+)^k $ be the area vector of $ D $. A symplectic divisor $ D $ is said be \textbf{concave} (resp. \textbf{convex}), if it satisfies positive (resp. negative) \textbf{GS criterion}, i.e. there exists $ z\in (\RR_+)^k $ (resp. $ (\RR_{\le 0})^k $) such that $ Q_D z=a $. A topological divisor $D=\cup C_i$ is called \textbf{non-negative} if $D\cdot C_i\ge 0$ for all $i$ and $D\cdot C_j\neq 0$ for some $j$. Similarly we can define a topological divisor to be \textbf{non-positive}, \textbf{positive} and \textbf{negative}.

For an $\omega$-symplectic divisor $D$, Gay-Stipsicz constructed in \cite{GaSt09} a convex plumbing for $D$ satisfying the negative GS criterion. This construction was extended to symplectic divisors satisfying the postive GS criterion in \cite{LiMa14-divisorcap}, where a concave plumbing is constructed for each such divisor. To summarize, we have the following theorem.
\begin{theorem}[\cite{GaSt09}, \cite{LiMa14-divisorcap}]\label{thm:divisorcap}
	Let $ D\subset (W,\omega ) $ be an $ \omega $-orthogonal symplectic divisor. Then $ D $ has a concave (resp. convex) plumbing if $ (D,\omega ) $ satisfies the positive (resp. negative) GS criterion.
\end{theorem}
We call this construction the GS construction and review it in Section \ref{section:GS}. Note that a different construction was presented in \cite{McL16-discrepancy}, which works in higher dimensions and does not require $\omega$-orthogonality.

The convex or concave plumbing depends on the symplectic divisor $(D,\omega)$ and other parameters, but the contact structure induced on the boundary depends only on the topological divisor $D$ (\cite{McL16-discrepancy}, \cite{LiMa14-divisorcap}). This motivates the notion of convex and concave topological divisors (see Section \ref{section:contact str}).
A much stronger uniqueness holds for convex divisor $D$, where the contact structure is called Milnor fillable. The Milnor fillable contact structure depends only on the oriented diffeomorphism type of $Y_D$ instead of the divisor $D$ (\cite{Caubel-Nemethi-Pampu2006}). In Section \ref{section:contact str} we formulate a suitable version of the following natural question.
\begin{question}
	Is there a similar unique contact sturcture on $-Y_D$ for concave $D$?
\end{question}
As a first step towards the question, we prove the contact structure $(-Y_D,\xi_D)$ is invariant under toric equivalence and interior blow-up/down of $D$ (Proposition \ref{prop:contact toric eq}) in the Appendix. Such invariance also plays an important role in the study of symplectic fillings of contact torus bundles in \cite{LiMaMi-logCYcontact}.

Furthermore, an open book decomposition was constructed on the contact boundary of the convex plumbing in \cite{GaSt09} for non-positive symplectic divisors. Later a Lefschetz fibration was constructed on the convex plumbing of a non-positive symplectic divisor in \cite{GaMa13-LF}. In Section \ref{section:openbook}, we extend the construction of \cite{GaSt09} to non-negative divisors.
\begin{prop}\label{construct-OBD}
	Let $D$ be a non-negative symplectic divisor and $(N_D,\omega)$ the concave plumbing constructed from the GS construction. Then there is an open book decomposition supporting the boundary $(-Y_D,\xi_D)$ of $(N_D,\omega)$. The page and monodromy of the open book decomposition can be read off directly from $D$.
\end{prop}

As an application of the open book decomposition, we investigate the universal tightness of some contact torus bundles. 
Let $X$ be a smooth rational surface and $D\subset X$ an effective reduced anti-canonical divisor. Such pair $(X,D)$ is called an anti-canonical pair and is related to Looijenga's conjecture on dual cusp singularities (\cite{Lo81}), which was proved in \cite{GrHaKe11} and later in \cite{En}.
Symplectic log Calabi-Yau pairs were then introduced in \cite{LiMa16-deformation} as a symplectic analogue of anti-canonical pairs. Enumerative aspects of symplectic log Calabi-Yau pairs and relations to toric actions were also studied in \cite{LiMiNi-toric}.

We call a topological divisor $ D $ consisting of a cycle of spheres a \textbf{circular spherical divisor}. A symplectic circular spherical divisor can be seen as a local version of a symplectic log Calabi-Yau pair as it doesn't require a closed ambient symplectic manifold.
It is well-known that the boundary of a plumbing of a cycle of spheres is a topological torus bundle (\cite{Ne81-calculus}). By Proposition 5.10 of \cite{LiMaMi-logCYcontact}, when $b^+(D)\ge 1$, $D$ admits a concave plumbing and its boundary $(-Y_D,\xi_D)$ is a contact torus bundle.

Golla and Lisca investigated a large family of such contact torus bundles in \cite{GoLi14}, determined their Stein fillability and studied the topology of Stein fillings. Then for all circular spherical divisors with $b^+(D)\ge 1$, the Stein fillability/non-fillability was determined and all minimal symplectic fillings were shown to have a unique rational homology type, by Mak and the authors (\cite{LiMaMi-logCYcontact}).


With their understanding of Stein fillings, Golla and Lisca showed in \cite{GoLi14} that a subfamily of the contact torus bundles they considered are universally tight. This led them to formulate the following conjecture.
\begin{conjecture}[\cite{GoLi14}]\label{conj}
	Suppose $ D $ is a circular spherical divisor with $ b^+(D)=1 $ and $ (-Y_D,\xi_D) $ symplectic fillable, then $ (-Y_D,\xi_D) $ is universally tight.
\end{conjecture}
This was confirmed for divisors with nonsingular intersection matrices by Ding-Li (\cite{DiLi18-torusbundle}). Both the results of Golla-Lisca and Ding-Li come from an extrinsic point of view and relies on understanding the symplectic fillings of virtually overtwisted contact torus bundles. 

We approach this conjecture from an intrinsic angle, based on the Giroux correspondence between contact structures and open book decompositions. 
Via the open book decomposition constructed in Proposition \ref{construct-OBD}, we combine the results of Honda (\cite{honda2000}) and Van-Horn-Morris (\cite{VHM-thesis}) to prove the following result in the direction of the above conjecture.
\begin{thm}\label{thm:universally-tight}
	Let $ D $ be a circular spherical divisor toric equivalent to a non-negative one, then $ (-Y_D,\xi_D) $ is universally tight, except possibly when $ -Y_D $ is a parabolic torus bundle with monodromy $ \begin{pmatrix}
		1 & n\\ 0 &1
		\end{pmatrix}, n>0 $.
\end{thm}
Because our approach is purely 3-dimensional in nature, our result is stronger than Conjecture \ref{conj} in the sense that we don't require $ D $ to be embedded in a rational surface. In fact, most contact structures we considered are not symplectic fillable, and thus cannot be studied by extrinsic methods. Also the circular spherical divisor $D$ we considered can have $b^+(D)\ge 2$, compared to $b^+(D)=1$ in Conjecture \ref{conj}.

%

\textbf{Acknowledgments}: The authors are grateful to Cheuk Yu Mak for useful discussions. Both authors are supported by NSF grant 1611680. 

\section{Local geometry of symplectic divisors}

\subsection{GS construction of divisor neighborhood}\label{section:GS}

We briefly review the construction of divisor neighborhood in \cite{GaSt09} and \cite{LiMa14-divisorcap}, i.e. the proof of Theorem \ref{thm:divisorcap}. The invariance of contact structure under blow-ups (Proposition \ref{prop:contact toric eq}) and the construction of open book decompositions (Proposition \ref{construct-OBD}) are based on this construction.


For each topological divisor $ D $, we can associate a decorated graph $ \Gamma=(V,E,g=(g_i),s=(s_i)) $ with each vertex $ v_i $ representing the embedded symplectic surface $ C_i $ and each edge connecting $ v_i,v_j $ corresponds to an intersection between $ C_i $ and $ C_j $. Each vertex $ v_i $ is weighted by the genus $ g_i=g(C_i) $ and self-intersection $ s_i=[C_i]^2 $. If $ (D,\omega) $ is a symplectic divisor, we can associate an augmented graph $ (\Gamma, a) $ by adding the area vector $ a=([\omega]\cdot [C_i])_{i=1}^k $.

For an augmented graph $ (\Gamma,a) $ and a vector $ z $ such that $ Q_\Gamma z=a $. Let $ z'=-\frac{1}{2\pi} z $ and fix a small $ \epsilon>0 $. For each vertex $ v_i $ and each edge $ e $ connecting to $ v_i $, we choose an integer $ s_{i,e} $ such that $ \sum_{e\in\mc{E}(v_i)} s_{i,e}=s_i $, where $ \mc{E}(v_i) $ denotes the set of edges $ e $ connecting to $ v_i $. Also, set $ x_{i,e}=-s_{i,e}z_i' - z_{j}' $, where $ v_j $ is the other vertex connected by $ e $.

Consider the first quadrant $ P=[0,\infty)^2\subset \RR^2 $ and for some fixed $ \gamma $ and $ \delta $, let $ g:P\to [0,\infty) $ be a smooth function with level sets like in the following figure. So $ g(x,y)=x $ when $ y-x>\gamma  $, $ g(x,y)=y $ when $ y-x<-\gamma  $ and $ g $ is symmetric with respect to the line $ y=x $.

\begin{figure}[h]
	\centering
	\includegraphics[scale=0.7]{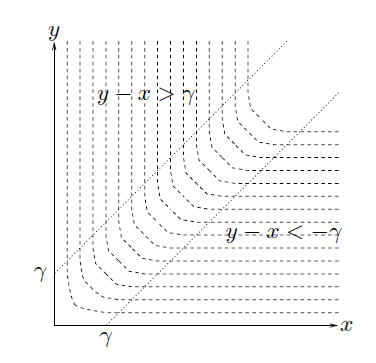}
	\caption{Contour of function $ g(x,y) $}
	\label{figure:g(x,y)}
\end{figure}

The constants $ \gamma  $ and $ \delta  $ are chosen to be small enough so that for each vertex $ v_i $ and each edge $ e $ incident to $ v_i $, the line passing through $ (0,\epsilon ) $ with tangent vector $ (1,-s_{i,e}) $ should intersect $ g^{-1}(\delta ) $ in the region $ y-x>\gamma  $. By symmetry, we also have the line passing through $ (\epsilon,0) $ with tangent vector $ (-s_{i,e},1) $ intersects $ g^{-1}(\delta) $ in the region $ y-x<-\gamma $.

For edge $ e $ connecting vertices $ v_i $ and $ v_j $, we can construct a local model $ (X_e,C_e,\omega_e,V_e,f_e) $ as follows. Let $ \mu:S^2\times S^2\to [z_i',z_i'+1]\times [z_j',z_j'+1] $ be the moment map of $ S^2\times S^2 $ onto its  image. We set $ p_1 $, $ p_2 $ be the coordinates for $ [z_i',z_i'+1]$, $ [z_j',z_j'+1] $ and set $ q_1,q_2\in \RR/2\pi\ZZ $ to be the corresponding fiber coordinates. Then $ \omega =dp_1\wedge dq_1 + dp_2\wedge dq_2 $ is the symplectic form on the preimage of the interior of the moment image. Let $ g_e(x,y)=g(x-z_i',y-z_j') $ and let $ R_e $ be the open subset of $ g_e^{-1}[0,\delta) $ between the line passing through $ (z_i',z_j'+2\epsilon) $ with tangent vector $ (1,-s_{i,e}) $ and the line passing through $ (z_j', z_i'+2\epsilon) $ with tangent vector $ (-s_{j,e},1) $. Let $ (X_e,\omega_e ) $ be the symplectic manifold given as the toric preimage $ \mu^{-1}(R_e) $. Let $ C_e=\mu_e^{-1}(\partial R_e) $, $ f_e=g_e\circ \mu_e $ and $ V_e $ be the Liouville vector field obtained by lifting the radial vector field $ p_1\partial_{p_1}+p_2\partial_{p_2} $ in $ \RR^2 $.

\begin{figure}[h] 
	\centering
	\begin{tikzpicture}[scale=1]
	\draw [<->,thick] (0,4.5) node (yaxis) [above] {$y$}
	|- (6,0) node (xaxis) [right] {$x$};
	\draw (1.5,0.9) coordinate (a_1) -- (4.5,0.9) coordinate (a_2);
	\draw (1.5,0.9) coordinate (b_1) -- (1.5,3.9) coordinate (b_2);
	\draw (3.2,1.8) coordinate (d_1)-- (5.4,1.8) coordinate (d_2);   
	\draw (2.4,2.6) coordinate (e_1)-- (2.4,3) coordinate (e_2); 
	\draw (d_1) to [bend left] (e_1);
	\draw (1.7,3.3) node [right] {$R_{i,e} $};
	\draw (2.4,3) coordinate (j_1)-- (2.4,4.3) coordinate (j_2);

	\draw (1.5,3.9) coordinate (f_1)-- (2.4,4.3) coordinate (f_2); 
	\draw (1.5,2.4) coordinate (g_1)-- (2.4,2.8) coordinate (g_2);
	\draw (3,0.9) coordinate (h_1)-- (3.9,1.8) coordinate (h_2);
	\draw (3.8,1.3) node [right] {$R_{j,e} $};
	\draw (4.5,0.9) coordinate (i_1)-- (5.4,1.8) coordinate (i_2); 
	\coordinate (c) at (intersection of a_1--a_2 and b_1--b_2);
	\draw[dashed] (yaxis |- c) node[left] {$z_{j}'$}
	-| (xaxis -| c) node[below] {$z_{i}'$};
	\draw[dashed] (1.5,2.4) -- (0,2.4) node[left] {$z_{j}'+\epsilon$};
	\draw[dashed] (1.5,3.9) -- (0,3.9) node[left] {$z_{j}'+2\epsilon$};
	\draw[dashed] (3,0.9) -- (3,0) node[below] {$z_{i}'+\epsilon$};
	\draw[dashed] (4.5,0.9) -- (4.5,0) node[below] {$z_{i}'+2\epsilon$};        
	
	\end{tikzpicture}
	
	\caption{Region $ R_e $}
	\label{fig GS construction}
\end{figure}
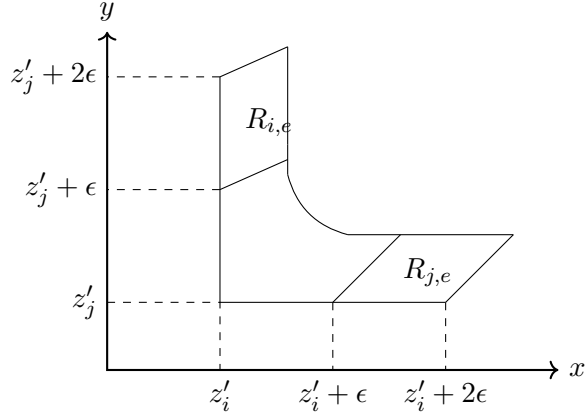

Then for each vertex $ v_i $ with valence $ d_i $, we may associate a 5-tuple $ (X_i,C_i,\omega_i,V_i,f_i) $ as follows. Let $ g_i $ be the genus of $ v_i $ and $ \Sigma_i  $ be a compact Riemann surface with genus $ g_i $ and $ d_i $ boundary components $ \partial_e\Sigma_i  $ corresponding to each edge $ e $ connected to $ v_i $. We can find a symplectic form $ \beta_i  $ and a Liouville vector field $ W_i $ on $ \Sigma_i $ such that there exists a collar neighborhood of $ \partial_e\Sigma_i  $ parametrized as $ (x_{i,e}-2\epsilon,x_{i,e}-\epsilon]\times S^1 $ on which $ \beta_i=dt\wedge d\alpha  $ and $ W_i=t\partial_t  $. Then we define $ X_i=\Sigma_i\times D^2(\sqrt{2\delta}) $ and $ \omega_i=\beta_i + rdr\wedge d\theta $, where $ D^2(\rho ) $ is the disk of radius $ \rho  $ and $ (r,\theta ) $ is the standard polar coordinate on the disk. We define $ f_i=\dfrac{r^2}{2} $, Liouville vector field $ V_i=W_i+(\dfrac{r}{2}+\dfrac{z_i'}{r})\partial_r $ and $ C_i=\Sigma_i-\partial \Sigma_i $.

Finally, the symplectic neighborhood $ (X,C,\omega,V,f) $ is constructed by gluing the local models together appropriately. Let $ R_{i,e} $ be the parallelogram in $ R_e $ cut out by the two lines with tangent vector $ (1,-s_{i,e}) $ passing through $ (z_i',z_j'+\epsilon) $ and $ (z_i',z_j'+2\epsilon) $ respectively. Similarly $ R_{j,e} $ is cut out by the two lines with tangent vector $ (-s_{j,e},1) $ passing through $ (z_j',z_i'+\epsilon) $ and $ (z_j',z_i'+2\epsilon) $ respectively. $ X_i $ can be glued to $ X_e $ by identifying $ \mu_e^{-1}(R_{i,e}) $ with $ (x_{i,e}-2\epsilon,x_{i,e}-\epsilon)\times S^1\times D^2(\sqrt{2\delta}) $. It's easy to check that symplectic forms, functions and Liouville vector fields all match accordingly.

It's easy to see that when $ (D,\omega) $ satisfies negative GS criterion, i.e. $ z\in (\RR_-)^k $, the Liouville vector field $ V $ points outward along the boundary. So the glued 5-tuple $ (X,C,\omega,V,f) $ gives the desired convex neighborhood. And when $ (D,\omega) $ satisfies positive GS criterion, we have $ z\in (\RR_+)^k $. Then we can choose $ t $ small enough such that $ V $ is inward pointing along the boundary of $ f^{-1}([0,t]) $, which gives a concave neighborhood. We would call this neighborhood the convex or concave plumbing of $ D $ and denote it by $ (N_D,\omega ) $.

In summary, given a symplectic divisor $(D,\omega)$ (or equivalently an augmented graph $ (\Gamma,a) $), a vector $ z $ satisfying positive/negative GS criterion and choices of parameters $ \epsilon ,  \delta ,t \in \RR_+ $, $ \{s_{v,e}\in \ZZ |\sum_{e\in \mc{E}(v)} s_{v,e}=s_v \} $, $ g:[0,\infty )^2\to [0,\infty )  $, the above construction gives a symplectic plumbing $ (N_D,\omega) $ with Liouville vector field $ V $ along the boundary.

Although the statement of Theorem \ref{thm:divisorcap} concerns an ambient symplectic manifold $(W,\omega)$, it actually only depends on the combinatorial data $(\Gamma,a)$. Suppose $ D $ is only a topological divisor with intersection matrix $ Q_D $ such that there exists $ z,a $ satisfying the positive (resp. negative) GS criterion $ Q_D z=a $. Then the GS construction actually constructs a compact concave (resp. convex) symplectic manifold $ (N_D,\omega_z) $ such that $ D $ is $ \omega_z- $orthogonal symplectic divisor in $ N_D $ and $ a $ is the $ \omega_z- $area vector of $ D $.

\subsection{Topological divisor and contact structure}\label{section:contact str}
Let $(N_D,\omega)$ be a symplectic plumbing of $D$ and $Y_D=\partial N_D$ be the oriented boundary 3-manifold of the plumbing $N_D$.
The Liouville vector field $V$ constructed above induces a contact structure $\xi_D=\ker (\alpha)$ on this boundary, where $\alpha=\iota_V\omega$. Note that when $N_D $ is convex (resp. concave), $\xi_D$ is a positive contact structure (i.e. $\alpha\wedge d\alpha>0$) on the oriented manifold $Y_D$ (resp. $-Y_D$).

The following uniqueness result implies that the symplectic structure $ \omega $ may vary but the induced contact structure on the boundary only depends on the topological divisor $ D $.

\begin{prop}[\cite{LiMa14-divisorcap}, cf. \cite{McL16-discrepancy}]\label{prop:unique-contact}
	Suppose $ D $ is an $ \omega- $orthogonal symplectic divisor which satisfies the positive/negative GS criterion. Then the contact structures induced on the boundary are contactomorphic, independent of choices made in the construction and independent of the symplectic structure $ \omega  $, as long as $ (D,\omega ) $ satisfies positive/negative GS criterion.
	
	Moreover, if $ D $ arises from resolving an isolated normal surface singularity, then the contact structure induced by the negative GS criterion is contactomorphic to the contact structure induced by the complex structure.
\end{prop}

This motivates us to consider the notion of convexity for topological divisors. A topological divisor $ D $ is called {\bf concave} (resp. {\bf convex}) if there exists $ z\in (\RR_+)^r $ (resp. $ z\in (\RR_{\le 0})^r $) such that $ a=Q_Dz\in (\RR_+)^r $. Then there is a contact manifold $ (-Y_D,\xi_D) $ (resp. $ (Y_D,\xi_D) $) and, for each choice of such $z$, a symplectic cap (resp. filling) $ (N_D,\omega_z) $ containing $ D $ as a symplectic divisor. 
One can check by simple linear algebra that being concave (resp. convex) is preserved by toric blow-up (see for example Lemma 3.8 of \cite{LiMa14-divisorcap}).
\begin{rmk}
	The notions of convex and concave for topological divisors are less restrictive than that for symplectic divisors, as we do not fix the symplectic area $a$.
\end{rmk}

When $ D $ is convex, $ (Y_D,\xi_D ) $ is contactomorphic to the contact boundary of some isolated surface singularity (\cite{Grauert1962}) and is called a Milnor fillable contact structure. A closed 3-manifold $ Y $ is called Milnor fillable if it carries a Milnor fillable contact structure. For every Milnor fillable $ Y $, there is a unique Milnor fillable contact structure (\cite{Caubel-Nemethi-Pampu2006}), i.e. the contact structure $ \xi_D  $ only depends on the oriented homeomorphism type of $ Y_D $ instead of $ D $ when $ D $ is convex. 

In light of this uniqueness result, it is natural to ask if similar results hold when $ D $ is concave. The answer is no and the following counterexample is given in \cite{LiMa14-divisorcap}.

\begin{example}[Example 2.21 of \cite{LiMa14-divisorcap}]
	Let $ D_1 $ be a single sphere with self-intersection $ 1 $ and $ D_2 $ be two spheres with self-intersections $ 1 $ and $ 2 $ intersecting at one point as follows.
	\[
	\begin{tikzpicture}
	\node (x) at (0,0) [circle,fill,outer sep=5pt, scale=0.5] [label=above: $ 1 $]{};
	\node (y) at (2,0) [circle,fill,outer sep=5pt, scale=0.5] [label=above: $ 2 $]{};
	\draw (x) to (y);
	\end{tikzpicture}
	\]
	Both divisors have a concave neighborhood. By \cite{Ne81-calculus} we can see that $ -Y_{D_1} $ and $ -Y_{D_2} $ are both orientation preserving homeomorphic to $ S^3 $. However, $ \xi_{D_1} $ is the unique tight contact structure on $ S^3 $ while $ \xi_{D_2}  $ is overtwisted.
\end{example}

So far all the counterexamples we can construct consist of divisors with different $ b^+ $ and also only one of them is fillable. So we refine our question to the following:
\begin{question}\label{question:divisor contact}
	Suppose $ D_1 $ and $ D_2 $ are concave divisors with $ -Y_{D_1}\cong -Y_{D_2} $. 
	Suppose either $ b^+(Q_{D_1})=b^+(Q_{D_2}) $ or $\xi_{D_1},\xi_{D_2}$ both symplectically fillable, then is $ (-Y_{D_1},\xi_{D_1}) $ contactomorphic to $ (-Y_{D_2},\xi_{D_2}) $?
\end{question}

We first introduce two operations on topological divisors.
\begin{definition}\label{def:toric eq}
	For a topological divisor $ D=\cup C_i $, a \textbf{toric blow-up} is the operation of adding a sphere component $E$ with self-intersection $ -1 $ between an adjacent pair of component $ C_i $ and $ C_{j} $, and changing the self-intersection of $ C_i $ and $ C_{j} $ by $ -1 $. \textbf{Toric blow-down} is the reverse operation.
	
	$ D^0 $ and $ D^1 $ are \textbf{toric equivalent} if they are connected by toric blow-ups and toric blow-downs. $ D $ is said to be \textbf{toric minimal} if no component is an exceptional sphere (i.e. a component of self-intersection $ -1 $).
\end{definition}
\begin{definition}\label{def:interior blow up}
	For a topological divisor $D=\cup C_i$, an \textbf{interior blow-up} is the operation of adding a sphere component $E$ with self-intersection $-1$ intersecting some component $C_i$ at one point, and changing the self-intersection of $C_i$ by $-1$. The reverse operation is called an \textbf{interior blow-down}.
\end{definition}
Since blow-ups and blow-downs can be performed in the symplectic category, these operations have symplectic analogues by adding an extra parameter of symplectic area. They will be described for augmented graphs in Section \ref{section:contact toric} and \ref{section:interior inv}.

Note that two divisors give the same oriented plumbed 3-manifold if and only if they are related by Neumann's plumbing moves (\cite{Ne81-calculus}), including toric blow-ups/blow-downs and interior blow-ups/blow-downs introduced above. To solve Question \ref{question:divisor contact}, it suffices to understand how the induced contact structure changes when we perform Neumann's plumbing moves.
As a first step towards this goal, we have the following proposition, whose proof is technical and thus deferred to the appendix.


\begin{prop}\label{prop:contact toric eq}
	The contact structure induced by the GS construction is invariant under \begin{enumerate}
		\item toric blow-ups/blow-downs,
		\item and interior blow-ups/blow-downs.
	\end{enumerate}
\end{prop}

In light of Proposition \ref{prop:contact toric eq}, we see that toric equivalence is a natural equivalence on divisors. For the study of contact structures and symplectic fillings, it suffices to consider toric minimal divisors. In particular it is used in the proof of Theorem \ref{thm:universally-tight}.


Note that all Milnor fillable contact structures are Stein fillable. Then we raise another question related to the fillability of divisor contact structures when $D$ is concave.
\begin{question}
	Is there a graph 3-manifold $Y$ such that $(-Y,\xi)$ is symplectically fillable for some contact structure $\xi$, but $-Y$ has no fillable divisor contact structure, i.e. for any concave $D$ with $Y_D=Y$, $(-Y,\xi_D)$ is not fillable?
\end{question}


\subsection{Non-negative divisor and open book decompositions}\label{section:openbook}
This subsection is devoted to the proof of Proposition \ref{construct-OBD}. We first recall some generalities on open book decompositions and refer the readers to \cite{etnyre-OBD} and \cite{ozbagci-stipsicz} for further details. An open book decomposition of a 3-manifold $ Y $ is a pair $ (B,\pi ) $ where $ B $ is an oriented link in $ Y $ such that $ \pi:Y\backslash B\to S^1 $ is a fiber bundle where the fiber $ \pi^{-1}(\theta)  $ is the interior of a compact surface $ \Sigma_\theta $ with boundary $ B $, for all $ \theta \in S^1 $. For each $ \theta\in S^1 $, $ \Sigma_\theta  $ is called a page while $ B $ is called the binding of the open book. An open book decomposition can also be described as $ (\Sigma,h) $ where $ \Sigma  $ is an oriented compact surface with boundary and $ h:\Sigma\to \Sigma  $ is a diffeomorphism such that $ h $ is identity in a neighborhood of $ \partial \Sigma  $. The map $ h $ is called the monodromy. 

An open book decomposition $ (B,\pi ) $ of a 3-manifold $ Y $ supports a contact structure $ \xi  $ on $ Y $ if $ \xi  $ has a contact form $ \alpha  $ such that $ \alpha(B)>0 $ and $ d\alpha (\Sigma)>0 $. Suppose we have an open book decomposition with page $ \Sigma  $ and monodromy $ h $. Attach a 1-handle to the surface $ \Sigma  $ along the boundary $ \partial \Sigma  $ to obtain a new surface $ \Sigma' $. Let $ \gamma  $ be a closed curve in $ \Sigma' $ transversely intersecting the cocore of this 1-handle exactly once. Define a new open book decomposition with page $ \Sigma' $ and monodromy $ h'=h\circ \tau_\gamma  $, where $ \tau_\gamma  $ denotes the right Dehn twist along $ \gamma  $. The resulting open book decomposition is called the a positive stabilization of the original one. The inverse of this process is called a positive destabilization. In \cite{giroux} Giroux established the one-to-one correspondence between oriented contact structures on $ Y $ up to isotopy and open book decompositions of $ Y $ up to positive stabilization. This correspondence is of fundamental importance and enables us to study contact structures through open book decompositions.

For the construction of open book decompositions, it is convenient to introduce the following notions. Let $D=\cup C_i$ be a topological divisor such that $D\cdot C_i\neq 0$ for some $i$. Then $D$ is called \textbf{non-negative} if $D\cdot C_i\ge 0$ for all $i$. Equivalently its associated decorated graph $\Gamma$ would satisfy $s_i+d_i\ge 0$ for all $i$ and $s_j+d_j\neq 0$ for some $j$, where $ d_i $ is the valence of vertex $ v_i $. Similarly we can define a topological divisor to be \textbf{non-positive}, \textbf{positive}, and \textbf{negative} in the obvious way. These notions were first introduced in \cite{etgu-ozbagci} and are intimately related to open book decompositions. Here is an easy observation.

\begin{lemma}\label{lemma:toric non-negative}
	A topological divisor being non-negative is preserved by toric blow-down.
\end{lemma}

It was shown in \cite{GaMa13-LF} that all non-positive divisors are actually negative definite and thus convex. For non-negative divisors, we can show the following:

\begin{lemma}\label{lemma:non-negative => positive GS}
	All non-negative divisors are concave.
\end{lemma}
\begin{proof}
	Let $ D  $ be a non-negative divisor with $ r $ components and denote by $ Q_D=(Q_{ij}) $ the intersection matrix. We will find a pair of vectors $ z $ and $ a $ through an iterated perturbation process.
	
	Start with $ z=(1,\dots,1)^T $ and $ a=Q_D z $. Since $ D  $ is non-negative, we have $ a_j\ge 0 $ for all $ j $ and $ a_i>0 $ for some $ i $. So the index set $ I=\{ i|a_i>0 \} $ is nonempty.
		
	Suppose $ a_l=0 $ and $ Q_{il}>0 $ for some $ i\in I $. Let $ z' $ be a new vector such that $ z_i'=z_i+\epsilon $ for some small positive $ \epsilon $ and $ z_j'=z_j $ for all other $ j $. Then we let $ a'=Q_D z' $ such that $ a'_j=a_j+\epsilon Q_{ji} $ for all $ j $. Since $ Q_{ji}\ge 0 $ for $ j\neq i $, we have $ a_j'\ge a_j $ for all $ j\neq i $. In particular, $ a_l'=a_l+\epsilon Q_{li}=\epsilon Q_{li}>0 $ as $ Q_{li}>0 $. For $ \epsilon $ small enough, we can also require that $ a_i'=a_i+\epsilon Q_{ii}>0 $. So we have $ I'=\{ i|a_i'>0 \}\supset I\cup \{ l \} $.
	
	We could repeat the process using $ I',z',a' $ as the new $ I,z,a $. Since the divisor $ D  $ is finite, this process stops at some finite time and produces a pair of vectors $ z ,a\in (\RR_+)^r $ such that $ Q_D z=a $.
\end{proof}

Based on their construction of convex divisor neighborhoods, Gay and Stipsicz constructed an open book decomposition supporting the induced contact structure on the boundary when the divisor is non-positive (\cite{GaSt09}). We first recall their construction and then extend it to the case of non-negative divisors.

\begin{lemma}[cf. \cite{GaSt09}]
	Let $ M^\pm =\pm [0,1]\times S^1\times S^1 $ with coordinates $ t\in [0,1],\alpha\in S^1 $ and $ \theta\in S^1 $. Given a nonnegative integer $ m $ there exists an open book decomposition $ \mbf{ob}^\pm_m=(B,\pi^\pm ) $ on $ M^\pm $ such that the following conditions hold:\begin{enumerate}
		\item $\pi^\pm |_{\{0\}\times S^1\times S^1} =\theta $
		\item $\pi^\pm |_{\{1\}\times S^1\times S^1} =\theta\pm m\alpha  $
		\item $ B $ has $ m $ components $ B_1,\dots,B_m $, which we take to be $ B_i=\{\dfrac{1}{2}\}\times \{\dfrac{2\pi i}{m} \}\times S^1 $
		\item The binding and pages can be oriented so that $ \pm \partial_\theta  $ is positively tangent to $ B_i $ and positively transverse to pages.
	\end{enumerate}
\end{lemma}

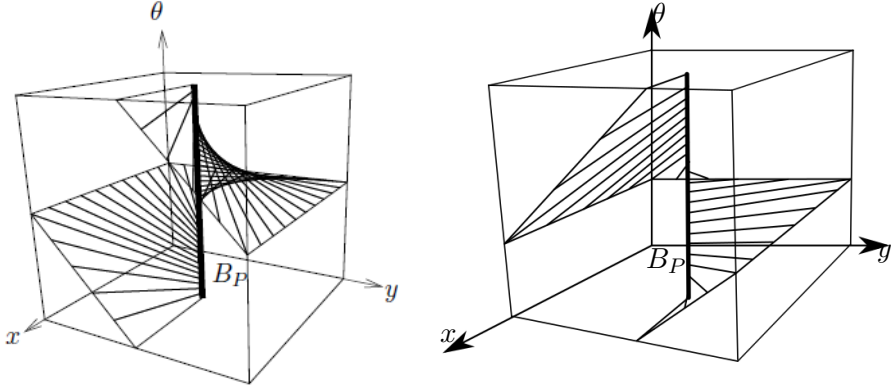
\begin{figure}[h]
	\begin{tikzpicture}
	\node[inner sep=0pt] (positive) at (0,0)
	{\includegraphics[width=.53\textwidth]{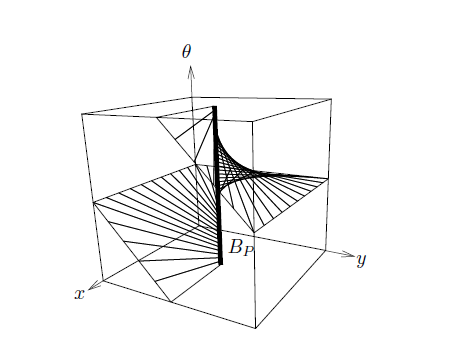}};
	\node[inner sep=0pt] (negative) at (6,0)
	{\includegraphics[width=.4\textwidth]{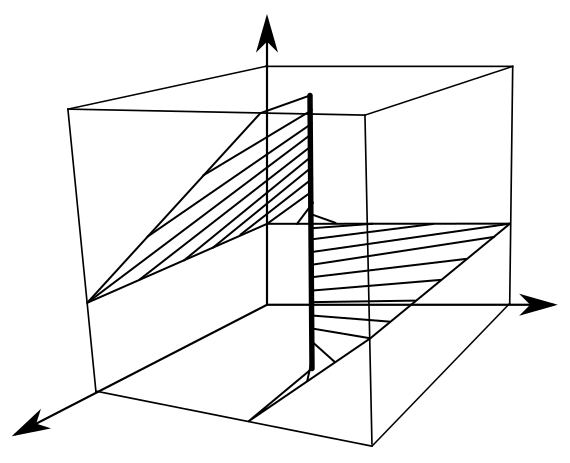}};
	\node[inner sep=0pt] (B) at (6,-1){$ B_P $};
	\node[inner sep=0pt] (x) at (3.1,-2){$ x $};
	\node[inner sep=0pt] (y) at (8.9,-0.9){$ y $};
	\node[inner sep=0pt] (theta) at (5.9,2.3){$ \theta $};
	\end{tikzpicture}
	\caption{Building blocks for the open book}
	\label{figure:openbook}
\end{figure}
\begin{proof}
	This lemma was proved in \cite{GaSt09} for $ M^+ $ only, where $ \mbf{ob}_m^+ $ is constructed by stacking $ m $ copies of building blocks with page shown on the left of Figure \ref{figure:openbook}. It's easy to see the same proof works for $ M^- $ by stacking copies of building blocks with page shown on the right of Figure \ref{figure:openbook}.
	
	The building blocks on the right are $P=-[0,1]\times [0,1]\times S^1$ with coordinates $(x,y,\theta)$. It is equipped with the open book decomposition $(B_P,\pi_P)$ satisfying $B_P=\{\dfrac{1}{2}\}\times \{\dfrac{1}{2}\}\times S^1$, $\pi|_{\{0\}\times [0,1]\times S^1}=\theta$, $\pi|_{ [0,1]\times \{0\}\times S^1}=\theta$, $\pi|_{ [0,1]\times \{1\}\times S^1}=\theta$ and $\pi|_{\{1\}\times [0,1]\times S^1}=\theta-2\pi y$. Note that when pages are oriented so that $-\partial_\theta$ is positively transverse, then $B_P$ is oriented (as boundary of the page) so that $-\partial_\theta$ is positively tangent.
\end{proof}
Recall that the boundary $ Y_D=\partial N_D $ is constructed by gluing $ f_i^{-1}(l) $ and $ f_e^{-1}(l) $ together if the edge $ e $ connects to vertex $ v_i $. For each vertex $ v_i $, we set the open book decomposition $\theta: f_i^{-1}(l)=C_i\times S^1_{\sqrt{2l}}\to S^1 $ to be the projection to second factor. For each edge $ e $, $ f^{-1}_e(l) $ is a submanifold with toric coordinates $ (p_1,q_1,p_2,q_2) $. We set the open book decomposition $ f^{-1}_e(l)\to S^1 $ to be $ q_1+q_2 $. Recall that each gluing region can be parametrized as $ (x_{i,e}-2\epsilon,x_{i,e}-\epsilon)\times S^1\times S^1_{\sqrt{2l}} $ with coordinates $ (t,\alpha,\theta) $ and $ q_1+q_2 $ transforms into $ (-s_{i,e}-1)\alpha +\theta  $. So if $ -s_i-d_i\ge 0 $, we can choose $ s_{i,e} $ so that $ p_{i,e}=-s_{i,e}-1 $ are all nonnegative. Then we can modify the open book decomposition on $ (x_{i,e}-2\epsilon,x_{i,e}-\epsilon)\times S^1\times S^1_{\sqrt{2l}} $ to be $ \mbf{ob}_{p_{i,e}}^+ $ and interpolate from $ q_1+q_2 $ to $ \theta  $.

Now we extend the construction to the concave case. The main difference from the convex case is that the open book decomposition supports the positive contact structure on the negative boundary $ -Y_D  $ of the concave neighborhood $ N_D $ instead of the positive boundary. So this open book is constructed by gluing $ \theta: -f_i^{-1}(l)= -\Sigma_i\times S^1_\rho \to S^1_\rho $ and $ q_1+q_2: -f_e^{-1}(l)\to S^1 $ together. Along the gluing region $ -(x_{i,e}-2\epsilon,x_{i,e}-\epsilon)\times S^1\times S^1_\rho  $ with coordinate $ (t,\alpha,\theta ) $, $ q_1+q_2 $ transforms to the function $ -(s_{i,e}+1)\alpha +\theta  $. We can modify the open book using the building block $ \mbf{ob}^-_{q_{i,e}} $ if $ q_{i,e}=s_{i,e}+1\ge 0 $. And such a choice of $ \{s_{i,e} \} $ exists if $ s_i+d_i\ge 0 $.

This open book decomposition is compatible with the canonical contact structure induced as boundary of the concave neighborhood. On $ -f_i^{-1}(l) $ the Reeb vector field is a negative multiple of $ \partial_\theta  $ and on $ -f_e^{-1}(l) $ the Reeb vector field is a negative multiple of $ b_1\partial_{q_1}+b_2\partial_{q_2} $ for some $ b_1,b_2>0 $. They are both positively transverse to the pages and positively tangent to the bindings.



For each vertex $ v_i $, let $ S_i $ be a compact surface with genus $ g_i $ and $ s_i+d_i $ boundary components. It's easy to see that the page $ S  $ of the above open book decomposition is given by connect-summing the surfaces $ S_i $ according to $ \Gamma  $. Let $ \{\gamma_1,\dots,\gamma_l \} $ be the collection of simple closed curves on $ S $ consisting of one circle around each connect-sum neck and $ \{\delta_1,\dots,\delta_q  \} $ be the collection of simple closed curves in $ S $ parallel to each boundary component. Here $ l=|E| $ and $ q=\sum_{i=1}^k(s_i+d_i)=\sum_{i=1}^ks_i + 2l $. For any simple closed curve $ c $ in $ S $, let $ \tau_c $ denote the right Dehn twist along $ c $. Then the monodromy is given by $ (\tau_{\gamma_1}\dots\tau_{\gamma_l})^{-1}(\tau_{\delta_1}\dots\tau_{\delta_q}) $. This finishes the proof of Proposition \ref{construct-OBD}.

\begin{example}
	The open book on the right of Figure \ref{figure:exampleOBD} corresponds to the divisor on the left. Here each vertex is decorated by $ (s_i,g_i) $ where $ s_i $ is the self-intersection number and $ g_i $ is the genus. Red curves are labeled with $ + $ or $ - $ to indicate that the monodromy consists of a positive or negative Dehn twist along the curve.
	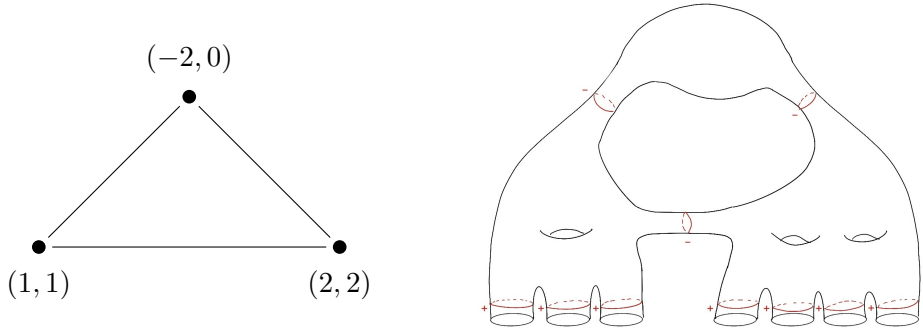
\begin{figure*}[h]
		\centering
		\begin{subfigure}{0.4\textwidth}
			\begin{tikzpicture}
			\node (x1) at (0,0) [circle,fill,outer sep=5pt,scale=0.5] [label=below:{$ (1,1) $}] {};
			\node (x2) at (2,2) [circle,fill,outer sep=5pt,scale=0.5] [label=above:{$ (-2,0) $}] {};
			\node (x3) at (4,0) [circle,fill,outer sep=5pt,scale=0.5] [label=below:{$ (2,2) $}] {};
			\draw (x1) to (x2);
			\draw (x2) to (x3);
			\draw (x1) to (x3);
			\end{tikzpicture}
		\end{subfigure}
		\begin{subfigure}{0.4\textwidth}
			\includegraphics[scale=0.25]{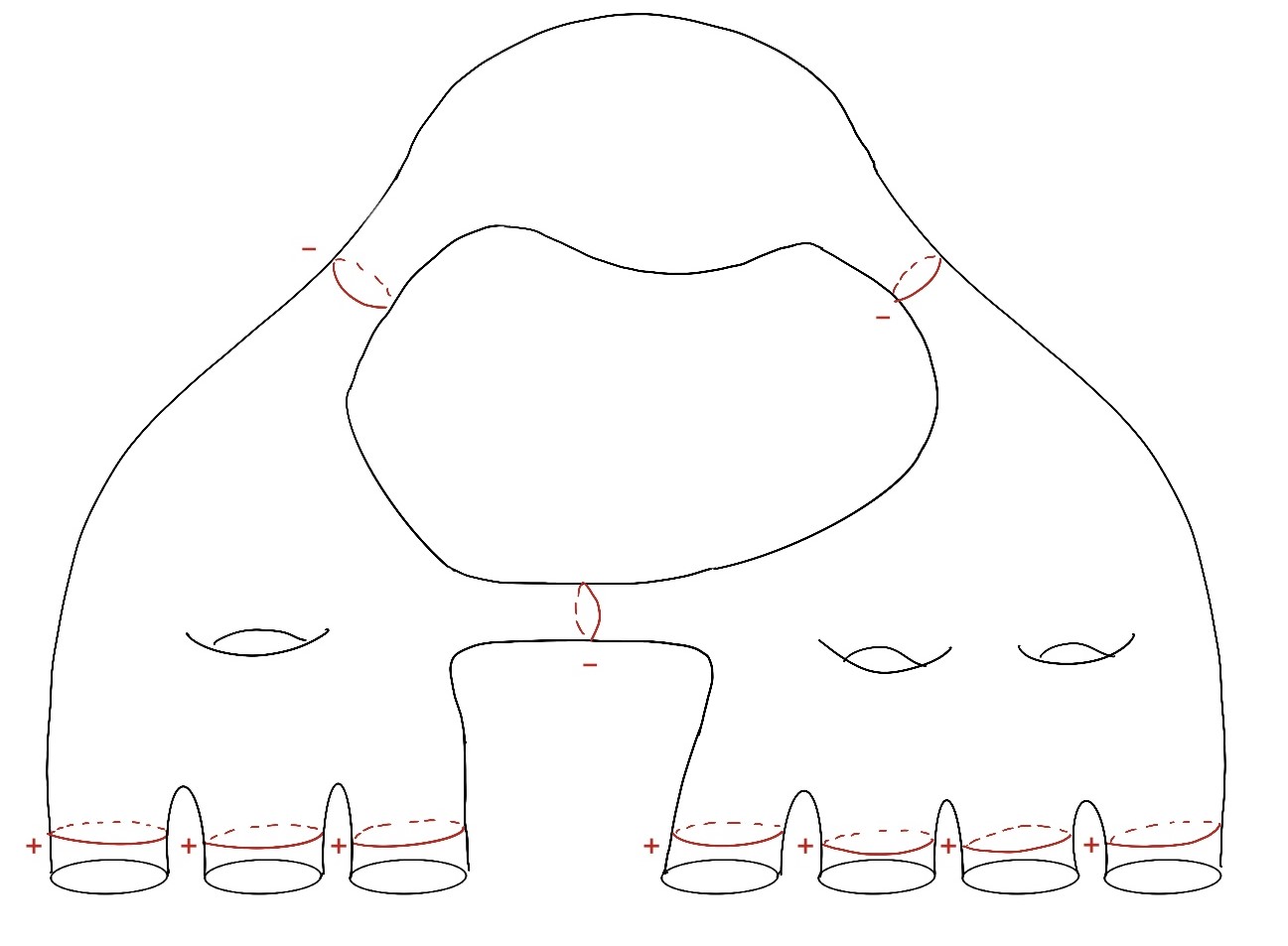}
		\end{subfigure}
		\caption{Divisor $ D $ (left) and open book decomposition for $ (Y_D,\xi_D) $}
		\label{figure:exampleOBD}
	\end{figure*}
	
\end{example}

\begin{rmk}
	The open book decomposition we constructed in the concave case matches the one constructed by Gay in \cite{gayOBD} and \cite{gayOBDerratum}. The construction of Gay makes use of handlebody theory and only works for positive divisors. Our construction is stronger as it works more generally for non-negative divisors.
	
	The open books constructed in both convex and concave cases match the ones constructed by Etgu and Ozbagci in \cite{etgu-ozbagci}, where the construction is purely topological and is not required to be compatible with a certain contact structure.
\end{rmk}

\section{Universally tight contact torus bundles}\label{application}


Honda has classified tight contact structures on torus bundles in \cite{honda2000}, which are mostly distinguished by their $ S^1- $twisting $ \beta_{S^1} $.  In his thesis (\cite{VHM-thesis}), Van-Horn-Morris described a correspondence between open book decompositions of tight contact torus bundles and word decompositions of their monodromies. Combining their results, we can determine that the contact torus bundles $ (-Y_D,\xi_D) $ are universally tight for a large family of circular spherical divisors $ D $.

Given a convex torus $ \Sigma=\RR^2/\ZZ^2 $ inside a tight contact manifold, its slope is the slope of a closed linear curve on $ \Sigma  $ that is parallel to a dividing curve. In this case, the dividing curves are parallel and homologically essential, so the slope is well-defined. To any slope $ s $ of a line in $ \RR^2 $ we can associate its angle $ \bar{\alpha}(s)\in \RR\PP^1=\RR/\pi\ZZ  $. For $ \bar{\alpha}_1,\bar{\alpha}_2\in \RR\PP^1 $, let $ [\bar{\alpha}_1,\bar{\alpha}_2] $ be the image of the interval $ [\alpha_1,\alpha_2]\subset\RR $, where $ \alpha_i\in \RR  $ are representatives of $ \bar{\alpha}_i $ and $ \alpha_1\le \alpha_2 <\alpha_1+\pi  $. A slope $ s $ is said to be between $ s_1 $ and $ s_0 $ if $ \bar{\alpha}(s)\in [\bar{\alpha}(s_1),\bar{\alpha}(s_0)] $.

Let $ \xi  $ be a contact structure on $ T^2\times I $ with convex boundary and has boundary slopes $ s_i $ for $ T^2\times \{ i\} $, $ i=0,1 $. $ \xi  $ is minimally twisting if every convex torus $ T\times t $ has dividing set with slope between $ s_1 $ and $ s_0 $. For a minimal twisting $ \xi  $, the $ I- $twisting of $ \xi  $ is given by $ \beta_I=\alpha_1-\alpha_0 $. For general $ \xi  $, cut $ (T^2\times I,\xi ) $ into minimally twisting segments $ T_k\cong T^2\times I_k $, $ k=1,\dots, l $ and its $ I- $twisting is the sum of each: $ \beta_I=\beta_{I_1}+\dots + \beta_{I_l} $. Then for a tight contact torus bundle $ M $, we define the $ S^1- $twisting $ \beta_{S^1} $ to be the supremum of the $ I- $twisting $ \lfloor\beta_I\rfloor  $ over all splittings of $ M $ into $ T^2\times I $ along a convex torus isotopic to a fiber, where $ \lfloor \beta_I \rfloor $ is defined to be $n\pi  $ if $ n\pi \le \beta_I < (n+1)\pi  $. $ (M,\xi ) $ is called minimally twisting in the $ S^1- $direction if $ \beta_{S^1}< \pi $.

Now we are ready to state Honda's result in the non-minimal twisting case.
\begin{prop}[Proposition 2.3 of \cite{honda2000}]\label{honda}
	For a torus bundle with monodromy $ A $, there exist infinitely many tight contact structures with non-minimal twisting. The universally tight contact structures are distinguished by the $ S^1 $-twisting $ \beta_{S^1} $ which take values in $ \{ m\pi | m\in \ZZ_+ \} $. There exists virtually overtwisted contact structure only when $ A=\begin{pmatrix}
	1 & n\\ 0 &1
	\end{pmatrix} $, $ n>1 $.
\end{prop}

So to decide whether a tight contact torus bundle is universally tight, it suffices to show it is non-minimal twisting, except the positive parabolic cases mentioned above. In order to calculate $ S^1 $-twisting from the divisor, we utilize the explicit open book decomposition for contact torus bundles described by Van-Horn-Morris. Let \textit{Word} denote the set of words in $ \{ a,a^{-1},b,b^{-1}  \} $. To $ a $ and $ b^{-1} $ we associate corresponding relative open book decompositions with pages and monodromies as in Figure \ref{figure:relativeOBD}. The relative open books for $ a^{-1} $ and $ b $ are the same as that  with sign reversed.
\begin{figure}[h]
	\includegraphics[scale=0.6]{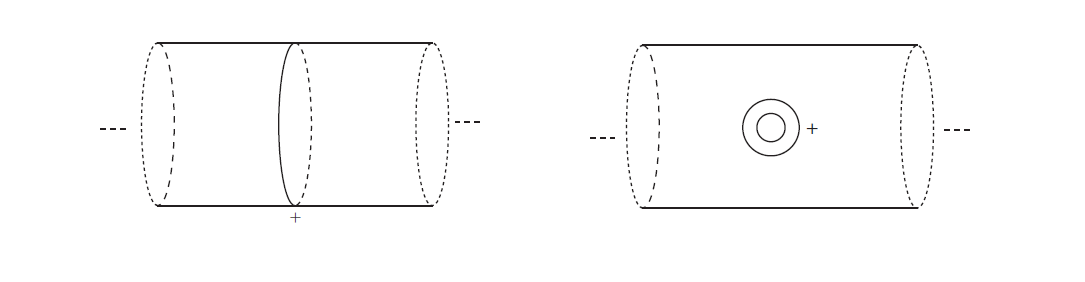}
	\centering
	\caption{Relative open book for $ a $ (left) and $ b^{-1} $ (right)}
	\label{figure:relativeOBD}
\end{figure}

To any word $ w\in \textit{Word} $, we can then associate an open book decomposition $ \mbf{ob}_w=(\Sigma_w,\phi_w) $ with torus pages $ \Sigma_w $ by stringing together the annular regions associated to each letter in $ w $ and identifying the remaining pair of circle boundaries to form a many-punctured torus as in Figure \ref{figure:torusOBD}. The monodromy $ \phi_w $ is given by Dehn twists along the all the signed curves. Denote the corresponding contact manifold by $ (Y_w,\xi_w ) $.
\begin{figure}[t]
	\includegraphics[scale=0.45]{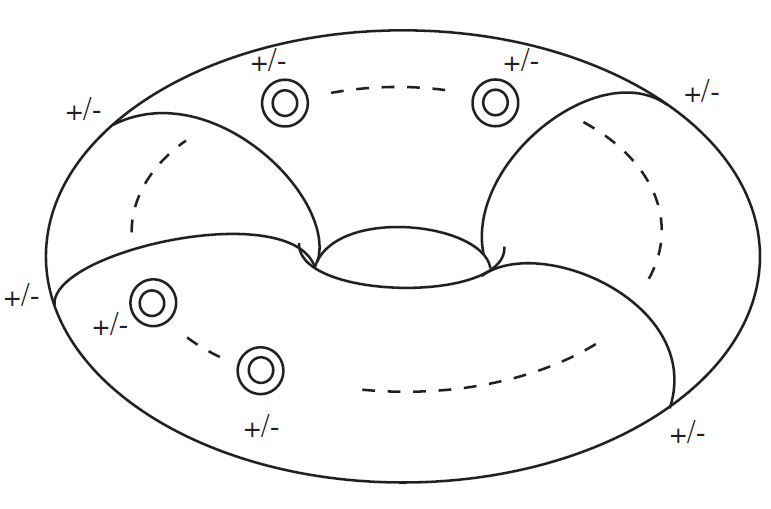}
	\centering
	\caption{Open book decomposition of torus bundles}
	\label{figure:torusOBD}
\end{figure}
\begin{lemma}[\cite{VHM-thesis}]
	Suppose words $ w,v $ are related by a sequence of braid relations $$ b^{-1}a^{-1}b^{-1}=a^{-1}b^{-1}a^{-1} .$$ Then the associated open book decompositions $ \mbf{ob}_w $ and $ \mbf{ob}_v $ are stably equivalent and thus their supported contact structures are isotopic.
\end{lemma}
Also it's clear that any cyclic permutation of the word does not change the associated open book decomposition at all. Adding canceling pairs of $ aa^{-1} $ or $ a^{-1}a $ leaves the page unchanged and only adds canceling pairs of Dehn twists $ \tau_c\tau_c^{-1} $ or $ \tau_c^{-1}\tau_c $ to the monodromy, which does not change the open book decomposition.

There is a natural map\[
\Phi:\textit{Word}\to Aut^+(T^2)\cong  \langle a^{-1},b^{-1}| b^{-1}a^{-1}b^{-1}=a^{-1}b^{-1}a^{-1} , (ab)^6=Id \rangle\cong  SL(2,\ZZ)
\]
defined by $ \Phi(a)=\begin{pmatrix}
1 & 1\\0 &1
\end{pmatrix} $, $ \Phi(b)=\begin{pmatrix}
1 & 0 \\ -1 & 1
\end{pmatrix} $, $ \Phi(a^{-1})=\Phi(a)^{-1} $ and $ \Phi(b^{-1})=\Phi(b)^{-1} $. Here $ Aut^+(T^2) $ is identified with $ SL(2,\ZZ ) $ by identifying $ T^2 $ with $ \RR^2/\ZZ^2 $. In the rest of this section, we will not distinguish between a word $ w $ and its image $ \Phi(w) $ in $ SL(2,\ZZ ) $ when we work with matrix multiplications. The 3-manifold supported by the open book decomposition $ \mbf{ob}_w $ is determined by the conjugacy class of $ \Phi(w) $ in $ SL(2,\ZZ ) $, but the contact structure varies when we take different words.

\begin{lemma}[\cite{VHM-thesis}]
	Let $ Y $ be the ambient manifold of the open book $ \mbf{ob}_w $. Then $ Y $ is homeomorphic to the torus bundle $ T_A$ with monodromy is $ A=\Phi(w) $.
\end{lemma}


Using the open book decomposition constructed in Proposition \ref{construct-OBD}, we can associated to any non-negative circular spherical divisor $ D $ a word $ w(D) $ that is solely composed of $ a^{-1},b^{-1} $ as follows. 
Recall in our construction of open book decompositions, each vertex $ v_i $ with self-intersection $ s_i $ contributes $ (s_i+d_i) $ boundary-parallel positive Dehn twists in the monodromy and each edge contributes a negative Dehn twist along the connect-sum neck. Thus each vertex corresponds to the word $ b^{-2-s_i} $ and each edge corresponds to the word $ a^{-1} $. The word $ w(D) $ is obtained by taking product of all these words in the clockwise order. Then the divisor $ D=(s_1,\dots,s_l) $ corresponds to the word $ w(D)=b^{-2-s_1}a^{-1}\dots b^{-2-s_l}a^{-1} $. The open book decomposition $ \mbf{ob}_{w(D)} $ associated to this word is exactly the one constructed in Proposition \ref{construct-OBD} for $ (-Y_D,\xi_D ) $. 
\begin{prop}[\cite{VHM-thesis}]\label{weakly-fillable}
	Any word in $ \{ a,a^{-1},b^{-1}\} $ gives an open book decomposition compatible with a weakly fillable contact structure.
\end{prop}
Since $ w(D) $ is a word in $ \{a^{-1},b^{-1} \} $, we have that $ (Y_D,\xi_D) $ is weakly fillable, and in particular, tight. We will usually write $ w $ for $ w(D) $ as our choice of divisor $ D $ would be clear from the context.

From Honda's classification, we know that if $ \beta_{S^1}\ge \pi $ and the monodromy is not conjugate to $ a^n,n>1 $, then the contact structure is universally tight. We can compute the $ S^1- $twisting of a contact structure from the word associated to its compatible open book decomposition as in \cite{VHM-thesis}. For a word $ w=a^{k_1}b^{-1}\dots a^{k_l}b^{-1} $, which we read from left to right as we move from $ t=0 $ to $ t=1 $. To compute the change of angles, we end with $ V_l=(1,0)^T $ and work backwards to $ t=0 $. Then $ V_{l-1}=a^{k_l}b^{-1}(1,0)^T $, so on and so forth. Let $ \mathfrak{c}_w $ denote the total angle change, then $ \beta_{S^1} $ of $ (Y_w,\xi_w) $ is at least $ \lfloor \mathfrak{c}_w \rfloor  $. Note that when calculating $ \beta_{S^1} $, we are free to change the word by braid relation, cyclic permutation and adding canceling pairs of $ a $ and $ a^{-1} $.

\begin{example}
	Consider parabolic bundle $ (-Y_D,\xi_D) $ given by the concave divisor in the following graph, with $ -2\le n $ so that the graph is non-negative. Its monodromy is $ A=A(-n,0)^{-1}=-\begin{pmatrix}
	1 & -n\\0 & 1
	\end{pmatrix} $.\begin{center}
		\begin{tikzpicture}
		\node (x) at (0,0) [circle,fill,outer sep=5pt,scale=0.5] [label=left:$ n $] {};
		\node (y) at (2,0) [circle,fill,outer sep=5pt,scale=0.5] [label=right:$ 0 $] {};
		\draw (x) to[bend left=30] (y);
		\draw (x) to[bend right=30] (y);
		\end{tikzpicture}
	\end{center}

	The word associated to this divisor is $ a^{-1}(b^{-1})^{n+2}a^{-1}(b^{-1})^2$.	Through cyclic permutation it becomes
	$w=b^{-1}a^{-1}(b^{-1})^{n+2}a^{-1}b^{-1}$.
	We can check that\[
	b^{-1}a^{-1}(b^{-1})^{n+2}a^{-1}b^{-1}\begin{pmatrix}
	1\\0
	\end{pmatrix}=\begin{pmatrix}
	-1\\0
	\end{pmatrix}.
	\]
	The rotation is $ \mathfrak{c}_w=\pi $. So $ \beta_{S^1}\ge \pi  $ and the contact structure is universally tight. 
\end{example}

\begin{proof}[Proof of Theorem \ref{thm:universally-tight}]
	We start by noticing that a word corresponding to a concave divisor $ D $ of length $l$ takes the form $ a^{-1}b^{-2-s_1}\dots a^{-1}b^{-2-s_l} $.
	By Proposition \ref{prop:contact toric eq}, we may assume $D$ is toric minimal or $D=(-1,p)$. Then $D$ is still non-positive by Lemma \ref{lemma:toric non-negative} and thus concave, we must have $b^+(Q_D)\ge 1$ by Proposition 5.12 of \cite{LiMaMi-logCYcontact}. Then we may further assume either $s_i\ge 0$ for some $i$, or $D=(-1,-2)$ or $(-1,-1)$, by Lemma 2.4 of \cite{LiMaMi-logCYcontact}.

	Assume $s_i\ge 0$ for some $i$. By cyclic permuting, we may assume $ s_l\ge 0 $ and the word becomes $w= b^{-1-s_l}\dots a^{-1}b^{-2-s_{l-1}}a^{-1}b^{-1} $ with $ -1-s_l\le -1 $. If $ -2-s_{l-1}=0 $, then $w$ can be written as the product $ w=w'b^{-1}a^{-n}b^{-1} $ with $ n\ge 2 $ for some word $ w' $ in $ \{ a^{-1},b^{-1} \} $. If $ -2-s_{l-1}\le -1 $, then $ w=w'b^{-1}a^{-l}b^{-m}a^{-1}b^{-1} $ with $ l\ge 1 $ and $ m\ge 0 $.
	
	
	The following direct computation shows that both $ b^{-1}a^{-n}b^{-1} $ and $ b^{-1}a^{-l}b^{-m}a^{-1}b^{-1} $ rotates the vector $ (1,0)^T $ by at least $ \pi  $:
	\begin{align*}
	b^{-1}a^{-n}b^{-1}\begin{pmatrix}
	1\\0
	\end{pmatrix}&=\begin{pmatrix}
	1-n \\ 2-n
	\end{pmatrix}, n\ge 2;\\ b^{-1}a^{-l}b^{-m}a^{-1}b^{-1}\begin{pmatrix}
	1\\0
	\end{pmatrix}&=\begin{pmatrix}
	-l \\ 1-l
	\end{pmatrix}, l\ge 1, m\ge 0.
	\end{align*}
	By adding canceling pairs of $ a $ and $ a^{-1} $, $ w' $ can always be written as product of $ (aba)^{-1} $, $ a $ and $ a^{-1} $. Note that $ a $, $ a^{-1} $ preserve the half space a vector sits in and $ (aba)^{-1} $ rotates a vector by $ \dfrac{\pi}{2} $ in the counterclockwise direction. As a result, $ w' $ does not rotate the vector back to the upper half space. So $ \mathfrak{c}_w\ge \pi  $ and $ \beta_{S^1}\ge \pi  $ for the contact structure induced on the boundary. By Proposition \ref{honda}, the contact structure is universally tight except when $ Y $ is torus bundle with monodromy $ A=a^n=\begin{pmatrix}
	1 & n\\ 0 &1
	\end{pmatrix},n>1 $. The remaining case of $ (-1,-1) $ and $ (-1,-2) $ follows from Proposition 4.1 of \cite{GoLi14}.
\end{proof}

\appendix
\section{Invariance of contact structure}
\subsection{Contact structure and toric equivalence}\label{section:contact toric}
In this section we prove the first statement of Proposition \ref{prop:contact toric eq}. We want to show that toric blow-up on the divisor doesn't change the induced contact structure on boundary of plumbing. The construction in this section will be adapted a little to prove the second statement of Proposition \ref{prop:contact toric eq} in the next section.

First we introduce the blow-up of an augmented graph, which is the symplectic version of toric blow-up. Consider the following local picture of an augmented graph (on the left), where each vertex is decorated by its self-intersection number, genus and symplectic area. The blow-up of this augmented graph with weight $ 2\pi a_0 $ is given on the right, which is the toric blow-up with areas specified in the graph. We call this an \textbf{augmented toric blow-up of edge $ e_0 $}. Similarly, the reverse operation is called an \textbf{augmented toric blow-down}.

\[
\begin{tikzpicture}[baseline=0ex]
\node (x) at (0,0) [circle,fill,outer sep=5pt, scale=0.5] [label=above:$ {(s_1,g_1,a_1)} $][label=below:$ v_1 $] {};
\node (y) at (2,0) [circle,fill,outer sep=5pt,scale=0.5] [label=above:$ {(s_2,g_2,a_2)} $][label=below:$ v_2 $] {};
\draw (x) to (y);
\node at (1,0) [label=below:$ e_0 $]{};
\end{tikzpicture}	
\quad\Longrightarrow\quad 
\begin{tikzpicture}[baseline=0ex]
\node (x) at (0,0) [circle,fill,outer sep=5pt, scale=0.5] [label=above:$ {(s_1-1,g_1,a_1-2\pi a_0)} $][label=below:$ v_1 $] {};
\node (y) at (3,0) [circle,fill,outer sep=5pt, scale=0.5] [label=above:$ {(-1,0,2\pi a_0)} $][label=below:$ v_0 $] {};
\node (z) at (6,0) [circle,fill,outer sep=5pt, scale=0.5] [label=above:$ {(s_2-1,g_2,a_2-2\pi a_0)} $][label=below:$ v_2 $] {};
\node at (1.5,0) [label=below:$ e_1 $]{};
\node at (4.5,0) [label=below:$ e_2 $]{};
\draw (x) to (y);
\draw (z) to (y);
\end{tikzpicture}	
\]
Denote the original augmented graph by $ (\Gamma^{(1)},a^{(1)}) $ and the blown-up graph $ (\Gamma^{(2)},a^{(2)}) $.  Note that $ Q_{\Gamma^{(2)}}z^{(2)}=a^{(2)} $ is still solvable after the augmented toric blow-up. If $ z^{(1)}=(z_1,z_2,\dots ) $ and $ a^{(1)}=(a_1,a_2,\dots ) $ satisfy $ Q_{\Gamma^{(1)}}z^{(1)}=z^{(1)} $, then after blow-up of area $ a_0 $, $ z^{(2)}=(z_1,z_1+z_2-2\pi a_0,z_2,\dots ) $ and $ a^{(2)}=(a_1-2\pi a_0,2\pi a_0,a_2-2\pi a_0,\dots) $ satisfy $ Q_{\Gamma^{(2)}}z^{(2)}=z^{(2)} $. So we could apply GS construction to both augmented graphs. In the following, we will denote the construction based on $ (\Gamma^{(1)},a^{(1)}) $ by GS-1 and denote the construction based on $ (\Gamma^{(2)},a^{(2)}) $ by GS-2. 

For the choice of $ \{s_{v,e} \} $, note that the two graphs differ only near $ e_0 $. 
We could choose $ \{ s_{v,e} \} $ for GS-1 first and then choose the same $ \{s_{v,e}\} $ for all vertices and edges for GS-2, except the ones involved in the toric blow-up. We could choose $ s_{v_0,e_1}=0$, $s_{v_0,e_2}=-1 $ so that $ s_{v_0,e_1}+s_{v_0,e_2}=s_0=-1 $ and choose $ s_{v_1,e_1}=s_{v_1,e_0}-1 $, $ s_{v_2,e_2}=s_{v_2,e_0}-1 $. Then we have $ x_{v_1,e_1}=x_{v_1,e_0}-a_0 $, $ x_{v_2,e_2}=x_{v_2,e_0}-a_0 $, $ x_{v_0,e_1}=-z_1' $ and $ x_{v_0,e_2}=z_1'+a_0 $. The choice of other parameters will be specified later. Note that the choice of parameters won't affect the boundary contact structure by Proposition \ref{prop:unique-contact}.



\begin{figure}[h!] 
	\centering
	\begin{tikzpicture}[scale=1.3]
	\draw [<->,thick] (0.5,4.5) node (yaxis) [above] {$y$}
	|- (6,0) node (xaxis) [right] {$x$};
	
	\draw (1.5,0.9) coordinate (a_1) -- (4.5,0.9) coordinate (a_2); 
	\draw (1.5,0.9) coordinate (b_1) -- (1.5,3.9) coordinate (b_2); 
	\draw (3.2,1.8) coordinate (d_1)-- (5.1,1.8) coordinate (d_2);   
	\draw (2.4,2.6) coordinate (e_1)-- (2.4,4.5) coordinate (e_2); 
	
	\draw (d_1) to [bend left] (e_1);
	
	\draw (b_2) -- (e_2);
	\draw (a_2) -- (d_2);
	\draw[dashed] (1.5,2.4) coordinate (g_1)-- (2.4,3) coordinate (g_2); 
	\draw[dashed] (3,0.9) coordinate (h_1)-- (3.6,1.8) coordinate (h_2); 
	
	\draw (3.7,1.3) node [right] {$R_{v_2,e_0}^{(1)} $};
	\draw (1.5,3.3) node [right] {$R_{v_1,e_0}^{(1)} $};
	
	\coordinate (c) at (intersection of a_1--a_2 and b_1--b_2);
	\draw[dashed] (yaxis |- c) node[left] {$z_{2}'$}
	-| (xaxis -| c) node[below] {$z_{1}'$};
	\draw[dashed] (1.5,2.4) -- (0.5,2.4) node[left] {$z_{2}'+\epsilon^{(1)}$};
	\draw[dashed] (1.5,3.9) -- (0.5,3.9) node[left] {$z_{2}'+2\epsilon^{(1)}$};
	\draw[dashed] (3,0.9) -- (3,0);\draw (2.6,0) node[below] {$z_{1}'+\epsilon^{(1)}$};
	\draw[dashed] (4.5,0.9) -- (4.5,0);\draw (4.8,0) node[below] {$z_{1}'+2\epsilon^{(1)}$ };        
	
	\draw[thick,red] (1.5,2.7) -- (3.3,0.9);
	\draw[dashed] (1.5,2.7) -- (0.5,2.7) node[left] {$ z_2'+a_0 $};
	\draw[dashed] (3.3,0.9) -- (3.3,0); \draw (3.6,0) node[below] {$ z_1'+a_0 $};
	
	\draw (e_2) node[right] {$ \begin{pmatrix}
		1\\ -s_{v_1,e_0}
		\end{pmatrix} $};
	\draw (d_2) node[above] {$ \begin{pmatrix}
		-s_{v_2,e_0}\\1
		\end{pmatrix} $};
	\end{tikzpicture}
	
	\caption{Region $ R_{e_0}^{(1)} $ with moment map $ \mu_{e_0} $}
	\label{fig: GS construction after blow-up}
\end{figure}
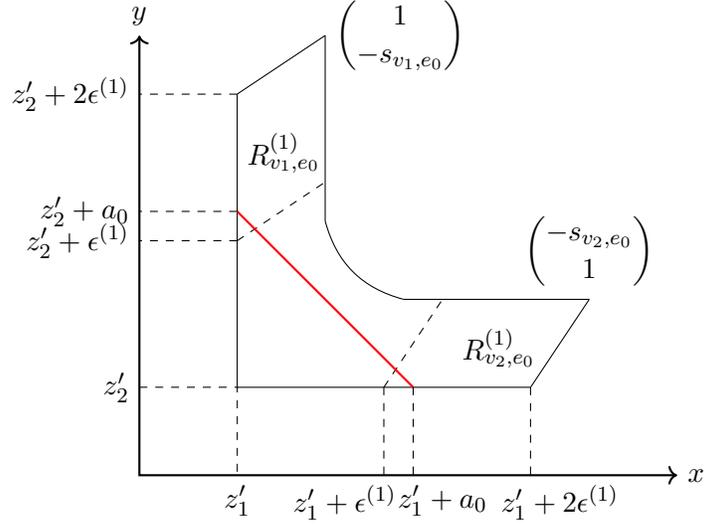

\begin{figure}[h!] 
	\centering
	\begin{subfigure}{0.4\textwidth}
		\begin{tikzpicture}[scale=1] 
		\draw [<->,thick] (0,4.5) node (yaxis) [above] {$y$}
		|- (6,0) node (xaxis) [right] {$x$};
		
		\draw (1.5,0.9) coordinate (a_1) -- (4.5,0.9) coordinate (a_2); 
		\draw (1.5,0.9) coordinate (b_1) -- (1.5,3.9) coordinate (b_2); 
		\draw (2.9,1.5) coordinate (d_1)-- (4.5,1.5) coordinate (d_2);   
		\draw (2.1,2.1) coordinate (e_1)-- (2.1,4.9) coordinate (e_2); 
		
		\draw (d_1) to [bend left] (e_1);
		
		\draw (b_2) -- (e_2);
		\draw (a_2) -- (d_2);
		\draw[dashed] (1.5,2.4) coordinate (g_1)-- (2.1,3.4) coordinate (g_2); 
		\draw[dashed] (3,0.9) coordinate (h_1)-- (3,1.5) coordinate (h_2); 
		
		\draw (3.5,1.2) node [right] {\scalebox{0.6}{$R_{v_0,e_1}^{(2)} $}};
		\draw (1.4,3.7) node [right] {\scalebox{0.6}{$R_{v_1,e_1}^{(2)} $}};
		
		\coordinate (c) at (intersection of a_1--a_2 and b_1--b_2);
		\draw[dashed] (yaxis |- c) node[left] {\scalebox{0.6}{$z_1'+z_{2}'+a_0$}}
		-| (xaxis -| c) node[below] {\scalebox{0.6}{$z_{1}'$}};
		\draw[dashed] (1.5,2.4) -- (0,2.4) node[left] {\scalebox{0.6}{$z_1'+z_{2}'+a_0+\epsilon^{(2)}$}};
		\draw[dashed] (1.5,3.9) -- (0,3.9) node[left] {\scalebox{0.6}{$z_1'+z_{2}'+a_0+2\epsilon^{(2)}$}};
		\draw[dashed] (3,0.9) -- (3,0) node[below] {\scalebox{0.6}{$z_{1}'+\epsilon^{(2)}$}};
		\draw[dashed] (4.5,0.9) -- (4.5,0) node[below] {\scalebox{0.6}{$z_{1}'+2\epsilon^{(2)}$}};        
		
		\draw (2.1,4.7) node[right] {\scalebox{0.6}{$ \begin{pmatrix}
				1\\ -s_{v_1,e_1}
				\end{pmatrix}=\begin{pmatrix}
				1\\ -s_{v_1,e_0}+1
				\end{pmatrix} $}};
		\draw (d_2) node[above] {\scalebox{0.6}{$ \begin{pmatrix}
				-s_{v_0,e_1}\\1
				\end{pmatrix}=\begin{pmatrix}
				0 \\1
				\end{pmatrix} $}};
		\end{tikzpicture}
		\caption{$ R_{e_1}^{(2)} $}
	\end{subfigure}
	\hfill
	\begin{subfigure}{0.4\textwidth}
		\begin{tikzpicture}[scale=1] 
		\draw [<->,thick] (0,4.5) node (yaxis) [above] {$y$}
		|- (6,0) node (xaxis) [right] {$x$};
		
		\draw (1.5,0.9) coordinate (a_1) -- (4.5,0.9) coordinate (a_2); 
		\draw (1.5,0.9) coordinate (b_1) -- (1.5,3.9) coordinate (b_2); 
		\draw (2.9,1.5) coordinate (d_1)-- (5.5,1.5) coordinate (d_2);   
		\draw (2.1,2.1) coordinate (e_1)-- (2.1,4.5) coordinate (e_2); 
		
		\draw (d_1) to [bend left] (e_1);
		
		\draw (b_2) -- (e_2);
		\draw (a_2) -- (d_2);
		\draw[dashed] (1.5,2.4) coordinate (g_1)-- (2.1,3) coordinate (g_2); 
		\draw[dashed] (3,0.9) coordinate (h_1)-- (4,1.5) coordinate (h_2); 
		
		\draw (3.8,1.1) node [right] {\scalebox{0.6}{$R_{v_0,e_2}^{(2)} $}};
		\draw (1.4,3.5) node [right] {\scalebox{0.6}{$R_{v_2,e_2}^{(2)} $}};
		
		\coordinate (c) at (intersection of a_1--a_2 and b_1--b_2);
		\draw[dashed] (yaxis |- c) node[left] {\scalebox{0.6}{$z_{2}'$}}
		-| (xaxis -| c);\draw (1.3,0) node[below] {\scalebox{0.6}{$z_1'+z_{2}'+a_0$}};
		\draw[dashed] (1.5,2.4) -- (0,2.4) node[left] {\scalebox{0.6}{$z_{2}'+\epsilon^{(2)}$}};
		\draw[dashed] (1.5,3.9) -- (0,3.9) node[left] {\scalebox{0.6}{$z_{2}'+2\epsilon^{(2)}$}};
		\draw[dashed] (3,0.9) -- (3,0) node[below] {\scalebox{0.6}{$z_1'+z_{2}'+a_0+\epsilon^{(2)}$}};
		\draw[dashed] (4.5,0.9) -- (4.5,0);\draw (5,0) node[below] {\scalebox{0.6}{$z_1'+z_{2}'+a_0+2\epsilon^{(2)}$}}; 
		
		\draw (e_2) node[right] {\scalebox{0.6}{$ \begin{pmatrix}
				1\\ -s_{v_0,e_2}
				\end{pmatrix}=\begin{pmatrix}
				1\\ 1
				\end{pmatrix} $}};
		\draw (d_2) node[above] {\scalebox{0.6}{$ \begin{pmatrix}
				-s_{v_2,e_2}\\1
				\end{pmatrix}=\begin{pmatrix}
				-s_{v_2,e_0}+1 \\1
				\end{pmatrix} $}};
		\end{tikzpicture}
		\caption{$ R_{e_2}^{(2)} $}
	\end{subfigure}
	\caption{Toric picture of edges $ e_1,e_2 $ in GS-2}
	\label{fig:two edges}
\end{figure}
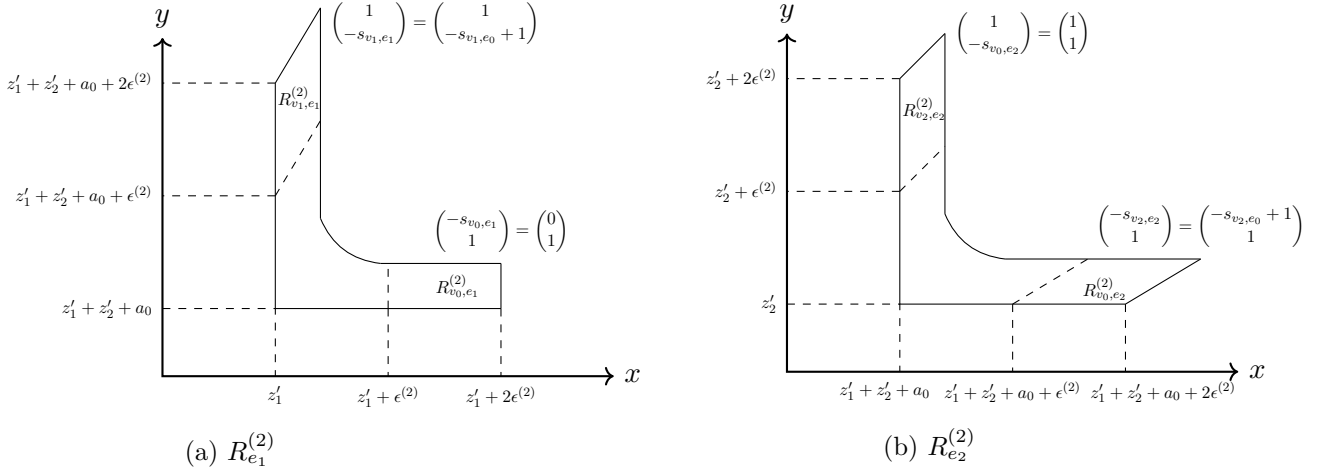

\begin{figure}[h!] 
	\centering
	\begin{tikzpicture}[scale=2]
	\draw [<->,thick] (0.5,4.5) node (yaxis) [above] {$y$}
	|- (6,0) node (xaxis) [right] {$x$};
	
	\draw (3,0.9) coordinate (a_1) -- (4.5,0.9) coordinate (a_2); 
	\draw (1.5,2.4) coordinate (b_1) -- (1.5,3.9) coordinate (b_2); 
	\draw (3.2,1.8) coordinate (d_1)-- (5.1,1.8) coordinate (d_2);   
	\draw (2.4,2.6) coordinate (e_1)-- (2.4,4.5) coordinate (e_2); 
	
	\draw[red,thick,dashed] (1.5,2.9) -- (1.8,3.1);
	\draw[red,thick,dashed] (3.5,0.9)--(3.7,1.2);
	
	\draw (d_1) to [bend left] (e_1);
	
	\draw (b_2) -- (e_2);
	\draw (a_2) -- (d_2);
	
	\draw (4,1) node [right] {\scalebox{0.6}{$R_{v_2,e_2}^{(1)} $}};
	\draw (1.45,3.5) node [right] {\scalebox{0.6}{$R_{v_1,e_1}^{(1)} $}};
	\draw (1.8,1.6) node [left] {\scalebox{0.6}{$ R_{v_0,e_1}^{(1)} $}};
	\draw (2.3,1.1) node [left] {\scalebox{0.6}{$ R_{v_0,e_2}^{(1)} $}};
	\draw (1.75,1.65)--(2.15,1.95); \draw (2.25,1.15)--(2.55,1.5);
	\draw (1.45,2.6) node[right] {\scalebox{0.8}{$ R_{e_1}^{(1)} $}};
	\draw (3.2,0.9) node [above] {\scalebox{0.8}{$ R_{e_2}^{(1)} $}};
	\draw (2.35,1.55) node [above] {\scalebox{0.8}{$ R_{v_0}^{(1)} $}};
	
	\draw[red,thick] (1.8,4.1)--(1.8,2.7);
	\draw[red,thick] (4.7,1.2)--(3.3,1.2);
	\draw[red,thick] (2,2.3)--(2.2,2.1);
	\draw[red,thick] (2.5,1.8)--(2.7,1.6);
	\draw[red,thick] (2.2,1.7)--(2.2,2.1);
	\draw[red,thick] (2.5,1.4)--(2.5,1.8);
	\draw[red,thick,dashed] (2.115,2.2)--(2.115,1.8);
	\draw[red,thick,dashed] (2.615,1.7)--(2.615,1.3);
	
	\draw[red,thick] (2,2.3) to [bend left] (1.8,2.7);
	\draw[red,thick] (3.3,1.2) to [bend left] (2.7,1.6);
	\draw[dashed] (2.1,2.2)--(2.6,1.7);
	\draw[dashed] (2.1,2.2)--(2.1,1.8);
	\draw[dashed] (2.6,1.7)--(2.6,1.3);
	
	\coordinate (c) at (intersection of a_1--a_2 and b_1--b_2);
	
	\draw[dashed] (1.5,3.9) -- (0.5,3.9) node[left] {$z_{2}'+2\epsilon^{(1)}$};
	\draw[dashed] (1.5,2.9) --(0.5,2.9) node[left] {$ z_2'+a_0+\epsilon^{(2)} $};
	
	\draw[dashed] (4.5,0.9) -- (4.5,0);\draw (4.8,0) node[below] {$z_{1}'+2\epsilon^{(1)}$};        
	\draw[dashed] (3.5,0.9)--(3.5,0);\draw (3.8,0) node[below] {$ z_1'+a_0+\epsilon^{(2)} $};
	
	\draw[thick] (1.5,2.4) -- (3,0.9);
	\draw[dashed] (1.5,2.4) -- (0.5,2.4 ) node[left] {$ z_2'+a_0 $};
	\draw[dashed] (3,0.9) -- (3 ,0 ); \draw (2.9,0) node[below] {$ z_1'+a_0 $};
	
	\draw (e_2) node[right] {$ \begin{pmatrix}
		1\\ -s_{v_1,e_0}
		\end{pmatrix} $};
	\draw (d_2) node[above] {$ \begin{pmatrix}
		-s_{v_2,e_0}\\1
		\end{pmatrix} $};
	\end{tikzpicture}
	
	\caption{Region $ R_{e_0}^{bl} $ after blow-up.}
	\small{
		$ R_{e_1}^{(1)} $ is the region on the upper left, enclosed by black and red solid lines.\\
		$ R_{e_2}^{(1)} $ is the region on the lower right, enclosed by black and red solid lines.\\
		$ R_{v_0}^{(1)} $ is the rectangular region in the middle, enclosed black dashed and solid lines.\\
		$ R_{v_0,e_1}^{(1)},R_{v_0,e_2}^{(1)} $ are the small rectangular regions in the middle bounded by both red and black lines.\\
		$ R_{v_1,e_1}^{(1)},R_{v_2,e_2}^{(1)} $ are the small parallelogram regions on the upper left and lower right.
	}
	\label{fig:blow-up together}
\end{figure}

\begin{figure}[h!] 
	\centering
	\begin{tikzpicture}[scale=4]
	\draw [<->,thick] (1.2,2.5) node (yaxis) [above] {$y$}
	|- (3,0.5) node (xaxis) [right] {$x$};
	\draw[thick] (1.5,1.9) -- (2.5,0.9);
	\draw[dashed] (1.5,1.9) -- (1,1.9 ) node[left] {$ z_2'+a_0 $};
	\draw[dashed] (2.5,0.9) -- (2.5 ,0.5);\draw (2.65,0.5) node[below] {$ z_1'+a_0 $};
	\draw[dashed] (2.3375,1.4775)--(1.6875,2.1275);
	\draw[dashed] (1.6875,2.1275)--(1.6875,1.7125);
	\draw[dashed] (2.3375,1.4775)--(2.3375,1.0625);
	\draw[red,thick] (1.875,1.95)--(1.875,1.525);
	\draw[red,thick] (2.175,1.65)--(2.175,1.225);
	\draw[red,thick] (1.6,2.225)--(1.875,1.95);
	\draw[red,thick] (2.5,1.325)--(2.175,1.65);
	
	\draw (1.9,2.4) node [right] {\scalebox{0.8}{$ R_{v_0,e_1}^{(1)} $}};
	\draw (2.3,2) node [right] {\scalebox{0.8}{$ R_{v_0,e_2}^{(1)} $}};
	\draw (2,1.45) node [above] {\scalebox{1}{$ R_{v_0}^{(1)} $}};
	\draw (1.95,2.35)--(1.8,1.8);\draw (2.35,1.95)--(2.25,1.4);

	\draw[dashed] (1.875,1.525) -- (1.2,1.525 ) node[left] {$ z_2'+a_0-2\epsilon^{(2)} $};
	\draw[dashed] (1.875,1.525) -- (1.875,0.3 ) node[below] {$ z_1'+2\epsilon^{(2)} $};
	\draw[dashed] (2.175,1.225) -- (1.2 ,1.225) node[left] {$ z_2'+2\epsilon^{(2)} $};
	\draw[dashed] (2.175,1.225) -- (2.175 ,0.5) node[below] {$ z_1'+a_0-2\epsilon^{(2)} $};
	\draw[dashed] (1.6875,1.7125) -- (1.2,1.7125 ) node[left] {$ z_2'+a_0-\epsilon^{(2)} $};
	\draw[dashed] (1.6875,1.7125) -- (1.6875,0.5 ) node[below] {$ z_1'+\epsilon^{(2)} $};
	\draw[dashed] (2.3375,1.0625) -- (1.2,1.0625 ) node[left] {$ z_2'+\epsilon^{(2)} $};
	\draw[dashed] (2.3375,1.0625) -- (2.3375,0.3 );\draw (2.4,0.3) node[below] {$ z_1'+a_0-\epsilon^{(2)} $};
	
	\end{tikzpicture}
	\caption{Zoomed picture of vertex region $ R_{v_0}^{(1)} $}
	\label{fig:vertex-exceptional sphere}
\end{figure}
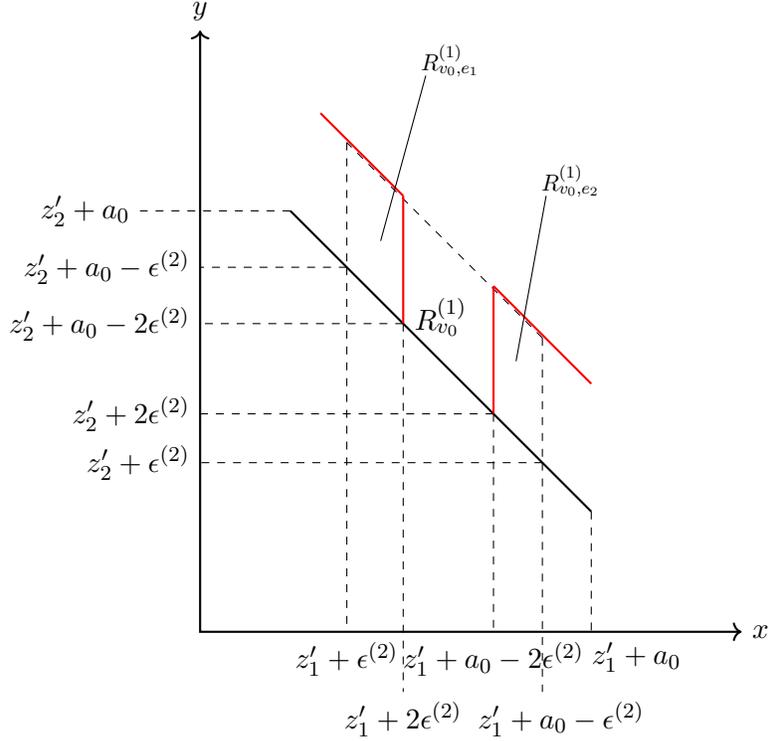

In GS-1, the edge $ e_0 $ corresponds to the local model $ (X_{e_0}^{(1)},C_{e_0}^{(1)},\omega_{e_0}^{(1)},V_{e_0}^{(1)},f_{e_0}^{(1)}) $ with toric image $ R_{e_0}^{(1)} $ in Figure \ref{fig: GS construction after blow-up}. The gluing region $ R_{v_1,e_0}^{(1)} $ is characterized by the vector $ \begin{pmatrix}
1 \\ -s_{v_1,e_0}
\end{pmatrix} $ and $ R_{v_2,e_0}^{(1)} $ is characterized by $ \begin{pmatrix}
-s_{v_2,e_0}\\ 1
\end{pmatrix} $.

In GS-2, the edge $ e_1 $ corresponds to the local model $ (X_{e_1}^{(2)},C_{e_1}^{(2)},\omega_{e_1}^{(2)},V_{e_1}^{(2)},f_{e_1}^{(2)}) $ with toric image $ R_{e_1}^{(2)} $ as in Figure \ref{fig:two edges}(a) with gluing region $ R_{v_1,e_1}^{(2)} $ characterized by vector $ \begin{pmatrix}
1\\ - s_{v_1,e_1}
\end{pmatrix}=\begin{pmatrix}
1\\ -s_{v_1,e_0}+1
\end{pmatrix} $ and $ R_{v_0,e_1}^{(2)} $ characterized by $ \begin{pmatrix}
-s_{v_0,e_1}\\1
\end{pmatrix} =\begin{pmatrix}
0 \\ 1
\end{pmatrix}$. Using the transformation $ \begin{pmatrix}
1 & 0\\-1 & 1
\end{pmatrix}\in GL(2,\ZZ ) $, we could map $ R_{e_1}^{(2)} $ onto $ R_{e_1}^{(1)}$ in Figure \ref{fig:blow-up together}. This gives a symplectomorphism $ \Phi_{e_1}:(\mu_{e_1}^{-1}(R_{e_1}^{(2)}),\omega_{e_1}^{(2)})\to  (\mu_{e_0}^{-1}(R_{e_1}^{(1)}),\omega_{e_0}^{(1)} )$ and identifies the Liouville vector field $ V_{e_1}^{(2)} $ with $ V_{e_0}^{(1)} $. Similarly, the edge $ e_1 $ corresponds to the local model $ (X_{e_2}^{(2)},C_{e_2}^{(2)},\omega_{e_2}^{(2)},V_{e_2}^{(2)},f_{e_2}^{(2)}) $ with toric image $ R_{e_2}^{(2)} $ as in Figure \ref{fig:two edges}(b) with gluing region $ R_{v_0,e_2}^{(2)} $ characterized by $ \begin{pmatrix}
1\\ - s_{v_0,e_2}
\end{pmatrix}=\begin{pmatrix}
1\\ 1
\end{pmatrix} $ and $ R_{v_2,e_2}^{(2)} $ by $ \begin{pmatrix}
-s_{v_2,e_2}\\1
\end{pmatrix} =\begin{pmatrix}
-s_{v_2,e_0}+1 \\ 1
\end{pmatrix}$. Using the transformation $ \begin{pmatrix}
1 & -1\\0 & 1
\end{pmatrix}\in GL(2,\ZZ ) $, we could map $ R_{e_2}^{(2)} $ onto $ R_{e_2}^{(1)}$ in Figure \ref{fig:blow-up together}. This gives symplectomorphism $\Phi_{e_2}: (\mu_{e_2}^{-1}(R_{e_2}^{(2)}),\omega_{e_2}^{(2)})\to  (\mu_{e_0}^{-1}(R_{e_2}^{(1)}),\omega_{e_0}^{(1)}) $, and identifies the Liouville vector field $ V_{e_2}^{(2)} $ with $ V_{e_0}^{(1)} $.

For vertex $ v_0 $, take $ X_{v_0}^{(2)}=[-z_1'-a_0+\epsilon^{(2)},-z_1'-\epsilon^{(2)}]\times S^1\times D^2_{\sqrt{2\delta^{(2)}}} $, $ \omega_{v_0}^{(2)}=dt\wedge d\alpha + rdr\wedge d\theta  $ and $ V_{v_0}^{(2)}=t\partial_t+(\dfrac{r}{2}+\dfrac{z_0'}{r})\partial_r $. So we see that the local model $ (X_{v_0}^{(2)},C_{v_0}^{(2)},\omega_{v_0}^{(2)},V_{v_0}^{(2)},f_{v_0}^{(2)}) $ is exactly $ (\mu_{e_0}^{-1}(R_{v_0}^{(1)}),\mu_{e_0}^{-1}(L),\omega_{e_0}^{(1)},V_{e_0}^{(1)},f_{e_0}^{(1)}) $, where $ L $ is the line segment from point $ (z_1'+\epsilon^{(2)},z_2'+a_0-\epsilon^{(2)}) $ to $ (z_1'+a_0-\epsilon^{(2)},z_2'+\epsilon^{(2)}) $ in Figure \ref{fig:vertex-exceptional sphere}. We can check the gluing of $ \mu^{-1}_{e_0}(R_{e_1}^{(1)}) $ with $\mu^{-1}_{e_0}( R_{v_0}^{(1)}) $ along $ \mu^{-1}_{e_0}(R_{v_0,e_1}^{(1)}) $ coincides with the gluing of $\mu^{-1}_{e_1}( R_{e_1}^{(2)}) $ with $ X_{v_0}^{(2)} $ along $ \mu^{-1}_{e_1}(R_{v_0,e_1}^{(2)}) $. Similarly, the gluing along $ \mu^{-1}_{e_0}(R_{v_0,e_2}^{(1)}) $ coincides with the gluing along $ \mu^{-1}_{e_2}(R_{v_0,e_2}^{(2)}) $. So the glued local model $ X_{e_1}^{(2)}\cup X_{v_0}^{(2)}\cup X_{e_2}^{(2)} $ is symplectomorphic to the preimage of the region $ R_{e_1}^{(1)}\cup R_{v_0}^{(1)}\cup R_{e_2}^{(1)} $ with Liouville vector fields identified.

Blow up the intersection point in $ P(D^{(1)}) $ corresponding to the edge $ e_0 $ symplectically with area $ 2\pi a_0 $ to get $ (P(D^{(1)})\# \overline{\CC\PP}^2,\omega^{bl}) $. This corresponds to cutting the corner from $ R_{e_0}^{(1)} $as shown in Figure \ref{fig GS construction} and the resulting region is called $ R_{e_0}^{bl} $. Since blowing up an interior point doesn't change the boundary, we have $ (Y_{D^{(1)}},\xi_{D^{(1)}})= \partial(P(D^{(1)}),\omega^{(1)}) \cong \partial(P(D^{(1)})\#\overline{\CC\PP}^2,\omega^{bl}) $.

To make the intervals in Figure \ref{fig:blow-up together} and Figure \ref{fig:vertex-exceptional sphere} well-defined, the following inequalities must be satisfied: \begin{align*}
2\epsilon^{(2)}<a_0-2\epsilon^{(2)} \text{ and } z_i'+a_0+2\epsilon^{(2)}\le z_i'+2\epsilon^{(1)}, i=1,2.
\end{align*}Also, in order for the embeddings and the blow-up to remain inside the neighborhood $ X^{(1)}_{e_0} $, the following restrictions on sizes of these neighborhoods should be satisfied:
\[
\delta^{(2)}\le \delta^{(1)} \text{ and } a_0<2\delta^{(1)}.
\]
So we could choose $ \delta^{(1)}, \delta^{(2)}, \epsilon^{(1)}, \epsilon^{(2)}, a_0 $ so that they satisfy $ 0<\delta^{(2)}\le \delta^{(1)}, 0<\epsilon^{(2)}<\epsilon^{(1)} $, $ a_0=2\epsilon^{(1)}-2\epsilon^{(2)} $ and $ 4\epsilon^{(2)}<a_0<2\delta^{(1)} $. Such choice of $ a_0 $ ensures that there is enough area to blow-up and the interval in $ X_{v_0}^{(2)} $ is well defined.
So the region  $ R_{e_1}^{(1)}\cup R_{v_0}^{(1)}\cup R_{e_2}^{(1)} $ is embedded in $ R^{(1)}_{e_0} $. Since all other local models are the same for GS-1 and GS-2, by shrinking the region $ R^{(1)}_{e_0} $, we get an contact isotopy from $ (Y_{D^{(1)}},\xi_{D^{(1)}})\cong \partial(P(D^{(1)})\#\overline{\CC\PP}^2,\omega^{bl}) $ to $  (Y_{D^{(2)}},\xi_{D^{(2)}}) =\partial (P(D^{(2)}),\omega^{(2)}) $.

\subsection{Half edge and blow-up of a vertex}\label{section:interior inv}
The construction outlined in Section \ref{section:GS} actually only works for graphs with at least two vertices, but it can be modified to take care of the single vertex case. Now consider the augmented graph $ (\Gamma ,a) $ where $ \Gamma  $ has only one vertex $ v $ decorated with genus $ g $ and self-intersection $ s $. As long as $ s\neq 0 $, there is always a solution $ z=\dfrac{a}{s} $.

According to GS construction, the vertex $ v $ corresponds to a local model $ (X_v,C_v,\omega_v,V_v,f_v) $. Here $ X_v=\Sigma_v\times D^2_{\sqrt{2\delta}}  $ where $ \Sigma_v $ is a genus $ g $ surface with one boundary component. To close up and get a disk bundle over a closed genus $ g $ surface with Euler class $ s $, we need to glue $ X_v $ to a disk bundle over disk and add the suitable twisting. Consider the region $ R_{\tilde{e}} $ in Figure \ref{fig:half edge}, which is similar to the region $ R_{e} $ in Figure \ref{fig GS construction} except we only have one gluing region $ R_{v,\tilde{e}} $. This region gives a local model $ (X_{\tilde{e}},C_{\tilde{e}},\omega_{\tilde{e}},V_{\tilde{e}},f_{\tilde{e}}) $ in the same way as the ordinary GS construction. Note that $ x_{\tilde{e}}=\mu_{\tilde{e}}^{-1}(R_{\tilde{e}})\cong D^2\times D^2_{\sqrt{2\delta}} $. Here the gluing region is specified by the vector $ \begin{pmatrix}
1 \\ -s
\end{pmatrix} $. By gluing these two local models, we get the desired disk bundle.

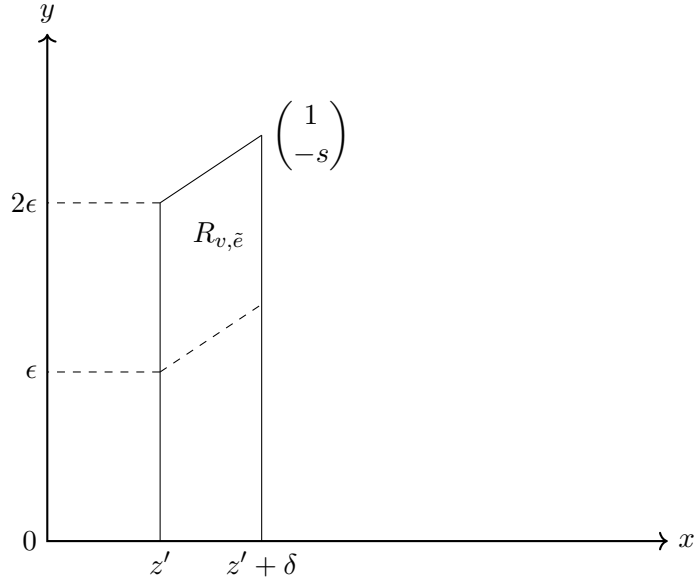
\begin{figure}[h] 
	\centering
	\begin{tikzpicture}[scale=1.5]
	\draw [<->,thick] (0.5,4.5) node (yaxis) [above] {$y$}
	|- (6,0) node (xaxis) [right] {$x$};
	
	\draw (1.5,0) coordinate (b_1) -- (1.5,3) coordinate (b_2); 
	\draw (2.4,0) coordinate (e_1)-- (2.4,3.6) coordinate (e_2); 
	
	
	\draw (b_2) -- (e_2);
	\draw[dashed] (1.5,1.5) coordinate (g_1)-- (2.4,2.1) coordinate (g_2); 
	
	\draw (1.7,2.7) node [right] {$R_{v,\tilde{e}} $};
	
	\draw (1.5,0) node[below] {$ z' $};
	\draw (0.5,0) node[left]{$ 0 $};
	\draw[dashed] (1.5,1.5) -- (0.5,1.5) node[left] {$\epsilon$};
	\draw[dashed] (1.5,3) -- (0.5,3) node[left] {$2\epsilon$};
	\draw (2.4,0) node[below] {$z'+\delta $};
	
	\draw (e_2) node[right] {$ \begin{pmatrix}
		1\\ -s
		\end{pmatrix} $};
	\end{tikzpicture}
	\caption{Region $ R_{\tilde{e}} $ corresponding to the half edge $ \tilde{e} $}
	\label{fig:half edge}
\end{figure}

This region $ R_{\tilde{e}} $ works almost the same as an edge in ordinary GS construction and we call it a \textbf{half edge}, as shown below.
\[
\begin{tikzpicture}[baseline=0ex]
\node (x) at (0,0) [circle,fill,outer sep=5pt, scale=0.5] [label=above:$ {(s,g,a)} $][label=below:$ v $] {};
\end{tikzpicture}	
\quad\Longleftrightarrow\quad 
\begin{tikzpicture}[baseline=0ex]
\node (x) at (0,0) [circle,fill,outer sep=5pt, scale=0.5] [label=above:$ {(s,g,a)} $][label=below:$ v $] {};
\node (y) at (1,0) {};
\node at (1,0) [label=below:$ \tilde{e} $]{};
\draw (x) to (y);
\end{tikzpicture}	
\]
For any vertex $ v $ in an augmented graph $ (\Gamma,a) $, we have $ X_v\cong \Sigma_v \times D^2 $. Take any point $ p\in \Sigma_v  $ and a small disk neighborhood $ D^2 $ of $ p $. This local neighborhood $ D^2\times D^2 $ can be regarded as the local model $ X_{\tilde{e}} $ corresponding to a half edge $ \tilde{e} $. Here we could choose the parameter $ s_{v,\tilde{e}}=0 $ so that $ s_{v,\tilde{e}}+\sum_{\mc{E}(v)} s_{v,e}=s_v $.
For an augmented graph $ (\Gamma,a) $, let $ v $ be a vertex in $ \Gamma $. We introduce the symplectic version of interior blow-up and blow-down. The following is called an \textbf{augmented interior blow-up} of vertex $ v $ with weight $ a_0 $, of which the reverse operation is also called an \textbf{augmented interior blow-down}.
	\[
	\begin{tikzpicture}[baseline=0ex]
	\node (x) at (0,0) [circle,fill,outer sep=5pt, scale=0.5] [label=left:$ {(s,g,a)} $][label=below:$ v $] {};
	\node (y) at (-1,1) [label=left:$ \dots $]{};
	\node (z) at (-1,-1) [label=left:$ \dots $]{};
	\draw (x) to (y);
	\draw (x) to (z);
	\end{tikzpicture}	
	\quad\Longrightarrow\quad 
	\begin{tikzpicture}[baseline=0ex]
	\node (x) at (0,0) [circle,fill,outer sep=5pt, scale=0.5] [label=left:$ {(s-1,g,a-a_0)} $] {};
	\node at (0,0.1) [label=below:$ v $]{};
	\node (y) at (-1,1) [label=left:$ \dots $]{};
	\node (z) at (-1,-1) [label=left:$ \dots $]{};
	\draw (x) to (y);
	\draw (x) to (z);
	\node (e) at (2,0) [circle,fill,outer sep=5pt, scale=0.5] [label=above:$ {(-1,0,a_0)} $][label=below:$ v_0 $] {};
	\draw (x) to (e);
	\node (f) at (1,0) [label=below:$ \tilde{e} $]{};
	\end{tikzpicture}	
	\]

An augmented interior blow-up of vertex $ v $ can be regarded as the augmented toric blow-up of a half edge $ \tilde{e} $ stemming from $ v $ as shown in the following diagram, where the right arrow indicates an augmented toric blow-up of $ \tilde{e} $.

\[
\begin{tikzpicture}[baseline=0ex]
\node (x) at (0,0) [circle,fill,outer sep=5pt, scale=0.5] [label=left:$ {(s,g,a)} $] {};
\node (y) at (-1,1) [label=left:$ \dots $]{};
\node (z) at (-1,-1) [label=left:$ \dots $]{};
\draw (x) to (y);
\draw (z) to (x);
\end{tikzpicture}	
\quad \Longleftrightarrow\quad	
\begin{tikzpicture}[baseline=0ex]
\node (x) at (0,0) [circle,fill,outer sep=5pt, scale=0.5] [label=left:$ {(s,g,a)} $] {};
\node (y) at (-1,1) [label=left:$ \dots $]{};
\node (z) at (-1,-1) [label=left:$ \dots $]{};
\draw (x) to (y);
\draw (z) to (x);
\node (e) at (2,0){};
\node at (1,0) [label=below:$ \tilde{e} $]{};
\draw (e) to (x);
\end{tikzpicture}\Longrightarrow\]
\[
\begin{tikzpicture}[baseline=0ex]
\node (x) at (0,0) [circle,fill,outer sep=5pt, scale=0.5] [label=left:$ {(s-1,g,a-a_0)} $] {};
\node (y) at (-1,1) [label=left:$ \dots $]{};
\node (z) at (-1,-1) [label=left:$ \dots $]{};
\draw (x) to (y);
\draw (z) to (x);
\node (e) at (1,0) [circle,fill,outer sep=5pt, scale=0.5] [label=above:$ {(-1,0,a_0)} $] {};

\draw (x) to (e);
\draw (e) to (2,0);
\node (f) at (1.5,0) [label=below:$ \tilde{e}' $]{};
\end{tikzpicture}	
\quad\Longleftrightarrow\quad
\begin{tikzpicture}[baseline=0ex]
\node (x) at (0,0) [circle,fill,outer sep=5pt, scale=0.5] [label=left:$ {(s-1,g,a-a_0)} $] {};
\node (y) at (-1,1) [label=left:$ \dots $]{};
\node (z) at (-1,-1) [label=left:$ \dots $]{};
\draw (x) to (y);
\draw (z) to (x);
\node (e) at (1,0) [circle,fill,outer sep=5pt, scale=0.5] [label=above:$ {(-1,0,a_0)} $] {};

\draw (x) to (e);

\end{tikzpicture}	
\]

The construction from Section \ref{section:contact toric} also works for regions like $ R_{\tilde{e}} $ with a suitable choice of $ \epsilon $ and $ a_0 $. So we have that toric blowing up a half edge $ \tilde{e} $ doesn't change its boundary contact structure. Thus we conclude that the boundary contact structure is invariant under interior blow-up of a vertex.

\bibliographystyle{plain}
\bibliography{ICCM}{}

\end{document}